\documentclass[11pt, reqno]{amsart}
\usepackage{amsthm}
\usepackage{amssymb,latexsym,graphics,enumerate}
\usepackage[mathscr]{eucal}
\usepackage{amsmath,amsfonts,amsthm,amssymb}
\usepackage[left=1in,right=1in,top=1in,bottom=1in]{geometry}
\usepackage{graphicx}
\usepackage[colorlinks=true,
            linkcolor=red,
            urlcolor=blue,
            citecolor=red]{hyperref}
\renewcommand{\geq}{\geqslant}

\renewcommand{\leq}{\leqslant}
\numberwithin{equation}{section}
\numberwithin{figure}{section}

\newcommand{\rr}{\mathbb{R}}

\newcommand{\be}{\begin{eqnarray*}}
\newcommand{\bel}{\begin{eqnarray}}
\newcommand{\ee}{\end{eqnarray*}}
\newcommand{\eel}{\end{eqnarray}}
\newcommand{\ba}{\begin{aligned}}
\newcommand{\ea}{\end{aligned}}
\newcommand{\de}{\Delta}

\newcommand{\na}{\nabla}
\newcommand{\ep}{\epsilon}
\newcommand{\ra}{\rightarrow}
\newcommand{\pa}{\partial}

\newcommand{\ds}{\displaystyle}
\newcommand{\n}[1]{|#1|}
\def\hf{\frac{1}{2}}
\def\x2+{{y}_+}
\def\y+{{y}_+}
\def\M+{M_+}
\def\V+{V_+}

\def\G+{\Gamma_+}
\def\Gm{\Gamma_-}
\def\Gc{\Gamma_0}
\def\S+{S_+}
\def\Sm{S_-}
\def\Sc{S_0}
\def\Gw{\Gamma_\omega}

\def\T{{\mathcal I}}  
\newcommand{\bb}{{\mathbf b}}
\newcommand{\myr}[1]{}
\newcommand{\myb}[1]{}
\newtheorem{thm}{Theorem}[section]

\newtheorem{defn}{Definition}[section]
\newtheorem{lem}{Lemma}[section]
\newtheorem{pro}{Proposition}[section]

\theoremstyle{remark}
\newtheorem{rmk}{Remark}[section]

\begin{document}

\title[Suppressing chemotactic blow-up through a fast splitting scenario]{Suppressing chemotactic blow-up through\\a fast splitting scenario on the plane}

\author{Siming He}
\address{Department of Mathematics and Center for Scientific Computation and Mathematical Modeling (CSCAMM), University of Maryland, College Park}
\email{simhe@cscamm.umd.edu}

\author{Eitan Tadmor}
\address{Department of Mathematics, Center for Scientific Computation and Mathematical Modeling (CSCAMM), and Institute for Physical Sciences \& Technology (IPST), University of Maryland, College Park}
\email{tadmor@cscamm.umd.edu}

\date{October 21, 2017} 




\begin{abstract}
We revisit the question of global regularity for the  Patlak-Keller-Segel (PKS) chemotaxis model. The classical 2D parabolic-elliptic model  blows up for initial mass $M>8\pi$. We consider more realistic scenario which takes into account the flow of the ambient environment induced by \emph{harmonic} potentials, and thus retain the identical elliptic structure as in the original PKS. Surprisingly, we find that already  the simplest case of \emph{linear} stationary vector field, $Ax^\top$, with large enough amplitude $A$, prevents chemotactic blow-up. Specifically, the presence of such an ambient fluid transport creates what we call a  'fast splitting scenario', which competes with the focusing effect of aggregation so that 'enough mass' is pushed away from concentration along the $x_1$-axis, thus avoiding  a finite time blow-up, at least for $M<16\pi$. Thus, the enhanced ambient flow  doubles  the amount of allowable mass which evolve to global smooth solutions.
\end{abstract}

\maketitle

\tableofcontents

\section{Introduction}
The Patlak-Keller-Segel (PKS) model describes the time evolution of  colony of bacteria with density $n(x,t)$ subject to two competing mechanisms --- aggregation triggered by the concentration of chemo-attractant driven by velocity field ${\mathbf u}:=\nabla(-\Delta)^{-1} n(\cdot,t)$, and  diffusion due to run-and-tumble effects,
\[
n_t + \nabla\cdot(n{\mathbf u})= \Delta n, \qquad {\mathbf u}:=\nabla (-\Delta)^{-1}n.
\]
We focus on the two-dimensional case where $(-\Delta)^{-1} n$ takes the general form of a fundamental solution together with an arbitrary harmonic function
\[
(-\Delta)^{-1}n(x,t) = (K*n)(x,t)+ H(x,t), \qquad K(x):= -\frac{1}{2\pi} \log |x|, \ \ \Delta H(\cdot,t)\equiv 0.
\]
The resulting PKS equation then reads
\bel\label{KS advection}
\frac{\pa n}{\pa t}+\na \cdot(n \na c)+\bb\cdot\na n=\de n,\quad\textcolor{blue}{c=K*n,} \quad x=(x_1,x_2)\in \rr^2,
\eel
subject to prescribed initial conditions $n(x,0)=n_0(x)$.
Here the divergence free vector field $\bb(\cdot)$  represents the environment of \textcolor{blue}{a} background fluid transported with velocity $\bb(x,t):=\nabla H(x,t)$. When $\bb\equiv 0$, the system is the classical parabolic-elliptic PKS equation modeling chemotaxis in a \emph{static} environment \cite{Pat1953,KS1970}.
We recall the large literature on the static case $\bb=0$, referring the interested reader to the review \cite{Hor2003} and the follow-up representative works \cite{Bil1995,BDP2006,BCM2008,BCC2008,CC2006,CC2008,CR2010,JL1992,BM2014,BK2010}. It is well-known that the large-time behavior of the static case \eqref{KS advection}${}_{\bb\equiv 0}$ depends on whether the initial total mass\footnote{We let $|x|$ denote the $\ell^2$-size of vector $x$ and let $\n{f}_p$ denote the $L^p$-norm of a vector function $f(\cdot)$.}$M:=\n{n_0}_1$ crosses the critical threshold of  $8\pi$: the equation admits global smooth solution in the  sub-critical case  $M<8\pi$ and it experiences a finite time blow-up if $M>8\pi$ \cite{BDP2006} (when $M=8\pi$,   aggregation and diffusion exactly balance each other and  solutions with finite second moments form Dirac mass as time approach infinity \cite{BCM2008}).

In this paper we study a more realistic scenario  of the PKS  model \eqref{KS advection}  where we take into account an ambient environment due to the fluid transport by vector field $\bb(\cdot,t)$.  Surprisingly, we find that already  the simplest case of linear stationary vector field, $\bb=A (-x_1,x_2)$, corresponding to $H(x)=\hf A(x^2_2-x^2_1)$, prevents chemotactic blow-up for $M< 16\pi$. As we shall see, the presence of such an ambient fluid transport creates what we call a  'fast splitting scenario'
which competes with the focusing effect of aggregation so that 'enough mass' is able to escape a finite time blow-up, at least for $M<16\pi$. This scenario is likely to be enhanced even further when larger amount of mass can be transported by a more pronounced ambient field $\bb(x,t)=\nabla H(\cdot,t) \gg |x|^q $ at $|x|\gg 1$.

We mention two other scenarios of non-static PKS with strong enough transport preventing blow-up for $M>8\pi$. In \cite{KX2015} the author exploit
relaxation enhancing of a vector field $\bb$ with a large enough amplitude in order to enforces global smooth \textcolor{blue}{solutions}. Here regularity follows due to a \emph{mixing property} over $\mathbb{T}^2$ and $\mathbb{T}^3$.  In \cite{BH2016} it was shown  how to exploit the enhanced dissipation effect of    non-degenerate \emph{shear flow} with large enough amplitude, \cite{BCZ2015}, in order to suppress the blow up in \eqref{KS advection} on $\mathbb{T}^2,\mathbb{T}^3,\mathbb{T}\times\rr,\mathbb{T}\times\rr^2$. It is worth mentioning that the model \eqref{KS advection} is one among many attempts to take into account the underlying fluid transport effect, see, e.g. \cite{L2010},\cite{L2012},\cite{LL2011},\cite{DLM2010},\cite{FLM2010}.

\subsection{A fast splitting scenario}
Here, we exploit yet another mechanism that suppress the possible chemotactic blow up of the equation \eqref{KS advection}, where the underlying fluid flow splits the population of bacteria with density $n$ exponentially fast,  resulting in several isolated subgroups with mass smaller than the critical $8\pi$. In this manner, an  initial total mass  greater than $8\pi$ is able to escape the finite-time concentration of Dirac mass. This provides a first no blow-up scenario over $\rr^2$, at least for $M$ up to $16\pi$.

We now fix the vector field driving a hyperbolic flow as the strain flow in \cite{KJCY2017}):
\bel\label{vector field b}
\bb(x):=A(-x_1,x_2).
\eel
Our aim is to show that  a large enough amplitude, $A\gg1$, guarantees the global existence of solution of  PKS \eqref{KS advection} subject to initial mass  $M<16\pi$.  Intuitively, the large enough amplitude $A\gg1$ is required so that  the ambient field $A(-x_1,x_2)$ \emph{`pushes away'  highly concentrated mass near the $x_1-$axis}, namely, $\ds \int_{|x_2|\leq \epsilon} n_0(x) dx \gg1$. With this we can state the main theorem of the paper.

\begin{thm}\label{thm 2}
Consider the PKS equation \eqref{KS advection},\eqref{vector field b} subject to  initial data, $n_0\in H^s(s\geq 2)$ with total mass, $M:=|n_0|_1< 16\pi$, such that $(1+|x|^2)n_0\in L^1(\rr^2)$ and $n_0\log n_0\in L^1(\rr^2)$. Assume $n_0$  is symmetric about the $x_1$-axis, and  that the ``$y$-component'' of its center of mass in the upper half plane
\begin{align*}
\x2+(t):=\frac{1}{\M+}\int_{x_2\geq 0}n(x,t)x_2dx,\quad \M+:=\int_{x_2\geq 0}n(x,t)dx \equiv \frac{M}{2},
\end{align*}
is not too close to the $x_1$-axis in the sense that \myr{$|x_2-\y+|^2$ or should we need $|x-(0,\y+)|^2$???}
\begin{equation}\label{initial configuration}
{\x2+^2}(0)> \frac{2}{\M+}\V+(0), \qquad \V+(t):=\int_{x_2\geq0}n(x,t)|x_2-\x2+|^2dx.
\end{equation}
Then there exists a large enough  amplitude, $A=A(M, \x2+(0), \V+(0))$, such that the weak solution of \eqref{KS advection},\eqref{vector field b} exists for all time and the free energy
\begin{align}
E[n](t)&:=\int \big(n\log n -\frac{1}{2}c n-H(x)n(x,t)\big) dx, \qquad H(x)=\frac{A}{2}(x_2^2-x^2_1), \label{E}
\end{align}
satisfies the dissipation relation
\begin{align}
E[n](t)&+\int_0^t\int_{\rr^2} n|\na \log n- \na c-\bb|^2dxds\leq E[n_0].\label{free energy estimate}
\end{align}
\end{thm}

We conclude the introduction with a series of remarks.
\begin{rmk}[Why large enough stationary field prevents second-moment collapse]
Our main theorem extends the amount of critical mass, so that global regularity of \eqref{KS advection},\eqref{initial configuration}
prevails for  $M<16\pi$, provided $A$ is large enough. To gain further insight why large enough $A$ will prevent blow-up, we recall that the blow up phenomena in the static case,  $\bb\equiv 0$  is deduced from the time evolution of the second moment, $\displaystyle V(t):=\int_{\rr^2}n(x,t)|x|^2dx$.
Indeed, a straightforward computation yields,  $\ds \dot{V}(t)=4M\Big(1-\frac{M}{8\pi}\Big)<0$, which implies that the positive $V(t)$ decreases to zero in a finite time and hence rules out existence of global classical solutions for $M>8\pi$.
In contrast, the  second moment of our non-static PKS equation \eqref{KS advection},\eqref{vector field b}, does not decrease to zero if $A$ is chosen large enough. This is the content of our next lemma.
\begin{lem}
Let $n(x,t)$ be the solution of  \eqref{KS advection} with vector $\bb(x)=A(-x_1,x_2)$, subject to initial data $n_0$ such that $\displaystyle W_0:=\int_{\rr^2}\!\!n_0(x)(x_2^2-x^2_1)dx$ is strictly positive. Then, if $A$ is chosen large enough, the second moment of the (classical) solution $\displaystyle V(t)= \int_{\rr^2}n(x,t)|x|^2dx$   increases in time.
\end{lem}
\begin{proof}
First, the time evolution of $V(t)$ can be calculated as follows:
\bel\label{time evolution of V}
\left\{
\begin{split}
\frac{d}{dt}V& =4M\left(1-\frac{ M}{8\pi}\right)+2\int x\cdot \bb\, n(x,t)dx\\
& =4M\left(1-\frac{ M}{8\pi}\right)+2A W, \qquad W(t)=\int_{\rr^2}(-x_1^2+x_2^2)n(x,t)dx.
\end{split}
\right.
\eel
Next, we compute the time evolution of $W(t)$
\be\ba
\frac{d}{dt}W
=&\int(-x_1^2+x_2^2)\na\cdot(\na n- \na c n-\bb n)dx\\
=&-\frac{1}{2\pi}\iint n(x,t)(-2x_1,2x_2)\cdot\frac{x-y}{|x-y|^2}n(y,t)dxdy+2AV\\
=&-\frac{1}{2\pi}\iint\frac{-(x_1-y_1)^2+(x_2-y_2)^2}{(x_1-y_1)^2+(x_2-y_2)^2}n(x,t)n(y,t)dxdy+2AV,
\ea\ee
where the last step follows by symmetrization.
Since the first term on the right is bounded from below by $-\frac{1}{2\pi}M^2$, we have
\bel\label{time evolution of W}
\frac{d}{dt}W\geq-\frac{1}{2\pi}M^2+2AV.
\eel
Finally, notice that since $W_0$ (and hence $V_0$) are assumed strictly positive, we can choose $A$ large such that
\bel\label{initial V W}
AV_0- \frac{1}{4\pi}M^2 \geq 0,\quad AW_0 + 2M\Big(1-\frac{M}{8\pi}\Big)\geq 0 .
\eel
Combining \eqref{initial V W}, \eqref{time evolution of V} and \eqref{time evolution of W} yields that ${W}(t) > 0,\enskip {V}(t)> 0$.
\end{proof}
This shows that collapse \underline{may} be prevented if $A$ is large enough. Indeed, our theorem \ref{thm 2} shows that collapse \underline{is} avoided for large enough $A$: specifically (consult \eqref{x2+ lower bound} below), $A={\M+}\delta^{-2}$ with small enough  $\delta$  such that
\[
\delta \leq(R-1)\sqrt{\frac{2\V+(0)}{\M+}}, \qquad R^2:= \M+\frac{{\y+^2}(0)}{2\V+(0)}.
\]
Observe that by \eqref{initial configuration} $R>1$, which enables the choice of $\delta \ll 1$ and hence $A\gg 1$ for the global regularity of \eqref{KS advection},\eqref{initial configuration}.\newline
\emph{We conjecture} that a similar fast splitting scenario holds with even higher-degree harmonics, $\bb(x,t)=\nabla H(\cdot,t) \gg |x|^q $ at $|x|\gg 1$ near the origin: the higher the degree $q$ is, the  larger amount of critical mass is \emph{expected} to be `pushed away' from the origin, with even higher thresholds than $16\pi$.
A main difficulty, however, is the lack of closure to control the second-order moments in these higher-order cases.
\end{rmk}

\begin{rmk}[On the free energy]
We note  that when $\bb=0$,  $E[n]$ becomes the classical dissipative free energy
\bel\label{ClassicalFreeEnergy}
\mathcal{F}=\int_{\rr^2}n\log ndx-\frac{1}{2}\int_{\rr^2}n cdx.
\eel
Due to the importance of the property \eqref{free energy estimate},  a weak solution of \eqref{KS advection} satisfying \eqref{free energy estimate} will be called a \emph{free energy solution}.
One of the  important properties of the PKS equation \eqref{KS advection} with background flow velocity \eqref{vector field b} is the dissipation of its free energy $E[n]$. The formal computation,  indicating the energy dissipation in non-static smooth solutions, is the content of our last lemma in this section.

\begin{lem}
Consider the PKS equation \eqref{KS advection} with background fluid velocity \eqref{vector field b}. If the solution is smooth enough, the free energy $E[n](t)$ is decreasing.
\end{lem}
\begin{proof}
The time evolution of the free energy \eqref{E} can be computed in terms of the potential $H=\hf A
 (x^2_2-x^2_1)$,
\be
\ba
\frac{d}{dt}E[n](t)
=&\int n_t(\log n- c-H)dx
=-\int n(\na \log n- \na c-\bb)\cdot(\na \log n-\na c-\na H)dx\\
=&-\int n|\na \log n- \na c-\bb|^2dx\leq 0.
\ea\ee
This completes the proof of the lemma.
\end{proof}
\end{rmk}

\begin{rmk}[Smoothness] Arguing along the lines \cite{EM2016}, one can prove that the free energy  solution admits higher-order integrability and consequently retains Sobolev smoothness for all positive time, $n\in C_c^{\infty}((0, T]; C^\infty_x)$ for all  $T<\infty$, thus our global weak solution is in fact a global strong solution.
\end{rmk}

\begin{rmk}[Stability]
We  note that by symmetry, our $16\pi$-threshold stated in theorem \ref{thm 2} is evenly divided between the upper- and lower-halves of the plane, each should contain at most $8\pi$ mass. We expect that this $16\pi$ threshold remains valid even for small \emph{asymmetric} perturbations, as long as  the mass in each half is kept \emph{separately} below the $8\pi$ threshold.
\end{rmk}

Our paper is organized as follows. In section 2, we introduce the regularized problem to (\ref{KS advection}) which leads to the local existence results. In section 3, we prove the main theorem, and in the appendix, we give detailed proofs of the results stated in section 2.

\section{Local existence}
\subsection{Weak formulation}
It is standard to understand the Keller-Segel equation \eqref{KS advection} with background fluid velocity \eqref{vector field b} in the following weak formulation.
\begin{defn}[weak formulation] The function $n$ is the weak solution of (\ref{KS advection}) if for $\forall \varphi\in C^\infty_{c}(\rr^2_+)$, the following holds:
\bel\label{weak formulation}
\ba
\frac{d}{dt}\int_{\rr^2}\varphi ndx=&\int_{\rr^2}\de \varphi ndx-\frac{1}{4\pi}\int_{\rr^2\times \rr^2}\frac{(\na\varphi(x)-\na \varphi(y))\cdot(x-y)}{|x-y|^2}n(x,t)n(y,t)dxdy\\
&+\int_{\rr^2}\na \varphi\cdot \bb ndx.
\ea
\eel
\end{defn}
Taking advantage of the assumed symmetry across the $x_1$-axis, one can further simplify the notion of a weak solution by restricting attention to the upper half plane, $\rr_+^2= \{(x_1,x_2)\,|\, x_2 \geq0\}$.
\begin{thm} The function $n$ is a weak solution of (\ref{KS advection}) if
$n_+:= n \mathbf{1}_{x_2\geq 0}$ satisfies for $\forall \varphi\in C^\infty_{c}(\rr^2_+)$,
\begin{eqnarray}\label{weak formulation for the upper half plane}
\ba
\frac{d}{dt}\int_{\rr_+^2}\varphi n_+dx=&\int_{\rr_+^2}\de \varphi n_+dx-\frac{1}{4\pi}\int_{\rr_+^2\times\rr_+^2}\frac{(\na\varphi(x)-\na \varphi(y))\cdot(x-y)}{|x-y|^2}n_+(x,t)n_+(y,t)dxdy\\
&+\int_{\rr^2_+}\na c_-\cdot \na \varphi n_+ dx+\int_{\rr_+^2}\na \varphi\cdot \bb n_+dx.
\ea
\end{eqnarray}
Here
\begin{align*}
\na c_-(x):=\textcolor{blue}{-\int_{(\rr_+^2)^c}\frac{x-y}{2\pi|x-y|^2}n_-(y)dy=-\int_{ (\rr_+^2)^c}\frac{x-y}{2\pi|x-y|^2}n_+(-y)dy}, \qquad
n_-:= n \mathbf{1}_{x_2\leq 0}.
\end{align*}
\end{thm}

\begin{proof} Rewrite (\ref{weak formulation}) as follows:
\be
\ba
\frac{d}{dt}\int_{\rr_+^2}\varphi n_+dx=&\int_{\rr_+^2}\de \varphi n_+dx-\frac{1}{4\pi}\int_{\rr_+^2\times\rr_+^2}\frac{(\na\varphi(x)-\na \varphi(y))\cdot(x-y)}{|x-y|^2}n_+(x,t)n_+(y,t)dxdy\\
&-\frac{1}{2\pi}\int_{\rr_+^2\times\rr^2_-}\frac{\na\varphi(x)\cdot(x-y)}{|x-y|^2}n_+(x,t)n_-(y,t)dxdy+\int_{\rr_+^2}\na \varphi\cdot \bb n_+dx.
\ea
\ee
The third term can be rewritten as $\displaystyle \int_{\rr^2_+}\na c_-\cdot\na \varphi\, n_+(x) dx$, and we get (\ref{weak formulation for the upper half plane}).
\end{proof}

\subsection{Regularized equation and local existence theorems}
In this section we introduce the local existence theorem and the blow up criterion for the Keller-Segel equation with advection. The theorems are standard, so the proofs are postponed to the appendix. The interested reader are referred to the papers \cite{BCM2008},\cite{BDP2006} for further details.

In order to prove the local existence theorem and the blow up criterion for the Keller-Segel system with advection (\ref{KS advection}), we regularize the system as follows:

\begin{equation}\label{rPKS}
\frac{\pa n^\ep}{\pa t}+\na\cdot (n^\ep\na c^\ep)+\bb\cdot\na n^\ep =\de n^\ep, \quad
 c^\ep:= K^\ep*n, \qquad x\in\rr^2, t>0,
\end{equation}
with the regularized kernel, $K^\ep$, given by
\bel\label{Kep}
K^\ep(z):=K^1\bigg(\frac{z}{\ep}\bigg)-\frac{1}{2\pi}\log \ep, \qquad K^1(z):=\left\{\begin{array}{ll}-\frac{1}{2\pi}\log|z|,&\text{ if}\quad|z|\geq 4,\\0,&\text{ if}\quad|z|\leq 1.\end{array}\right.
\eel
\ifx
the following estimates on $K^1$ hold:
\bel\label{K1 property}
|\na K^1(z)|\leq \frac{1}{2\pi|z|},\quad K^1(z)\leq -\frac{1}{2\pi}\log |z|
\eel
for any $z\in \rr^2$. As a result, we also have the following estimate on $K^\ep$:
\bel\label{na K ep property}
|\na K^\ep (z)|\leq \frac{1}{2\pi |z|}, \forall z\in \rr^2.
\eel
\fi%
Noting that $|\na K^\ep (z)|\leq C_\ep$ for all $z\in \rr^2$, it follows that the solutions to the equation \eqref{rPKS} exist for all time.  The proof is similar to the corresponding proof in the classical case. We refer the interested reader to the paper \cite{BDP2006} for more details.

Before stating the local existence theorems, we introduce the entropy of the solution
\bel\label{entropy}
S[n]:=\int_{\rr^2}n\log ndx.
\eel
Now the local existence theorems are expressed as follows:
\begin{pro}{$($Local Existence Criterion$)$}.\label{pro 2.1}
Assume that $|\bb|(x)\leq C |x|, \forall x\in \rr^2$. Suppose $\{n^\ep\}_{\ep\geq 0}$ are the solutions of the regularized equation \eqref{rPKS} on $[0,T^*)$. If $\{S[n^\ep](t)\}_\ep$ is bounded from above uniformly in $\ep$ and in $t\in [0,T^*)$, then the cluster points of $\{n^\ep\}_{\ep\ra0}$, in a suitable topology, are non-negative weak solutions of the PKS system with advection \eqref{KS advection} on $[0,T^*)$ and satisfies the relation \eqref{free energy estimate}.
\end{pro}

\begin{pro}\label{pro 1.2}{$($Maximal Free-energy Solutions$)$}. Assume the boundedness of the vector field \textcolor{blue}{$|\bb|(x)\leq C|x|$} and the integrability of initial data $(1+|x|^2)n_0\in L^1_+(\rr^2),\enskip n_0\log n_0\in L^1(\rr^2)$. Then there exists a maximal existence time $T^*>0$ of a free energy solution to the PKS system with advection \eqref{KS advection},\eqref{free energy estimate}. Moreover, if $T^*<\infty$ then
\be
\lim_{t\rightarrow T^*}\int_{\rr^2}n\log n dx=\infty.
\ee
\end{pro}
We conclude that if the entropy $S[n](t)=\int n\log n$ is bounded, then the free energy solution of \eqref{KS advection} exists locally. Moreover, if $S[n](t)<\infty $ for all $t<\infty$, the solution exists for all time.

\section{Proof of the main results}
\subsection{The three-step `battle-plan'} We proceed in three steps. The \emph{first step} carried in section \ref{sec:step1} below, is to control cell density distribution.
From the last section, we see that an entropy bound is essential for derivation of  local existence theorems for the PKS equation (\ref{KS advection}),\eqref{vector field b}. To this end, information about the distribution of cell density is crucial. The following lemma  is the key to the proof of the main results. It shows that mass cannot concentrate along the the $x_1$-axis, since we can find a thin enough strip along the $x_1$-axis with controlled amount of mass.

\noindent
For the rest of this section we fixed a small parameter $0<\eta\ll1$ which will quantify the sharp estimates in the sequel.

\begin{lem}\label{cell distribution lemma}
Suppose a sufficiently smooth $n_0$ is symmetric about the $x_1$ axis and assume that
\bel\label{initial configuration 1}
R^2:= \M+\frac{{\y+^2}(0)}{2\V+(0)} >1.
\eel
Fix a small enough $0<\eta\ll1$.
Then there exists  $\delta=\delta(\x2+(0), \V+(0), M, \eta)$ such that if we choose $A>\frac{\M+}{\delta^2}$, the smooth solutions to the regularized \eqref{rPKS}${}_\ep$ satisfy, uniformly for small enough $\ep$,
\bel\label{local mass estimate}
\int_{|x_2|\leq 2\delta} n^\ep(x,t) dx \leq \frac{(1+\eta)^2}{2R^2}M.
\eel
\end{lem}
 Condition \eqref{local mass estimate} implies, at least for $M<16\pi$, that the mass inside that $\delta$-strip is less than $8\pi$. On the other hand, it indicates the reason for the limitation $R>1$: for if $R<1$, then the bound \eqref{local mass estimate} would allow  a concentration of mass $\frac{M}{2R^2} \geq 8\pi$ inside the strip $|x_2|\leq 2\delta$, which in turn leads to a finite-time blow-up.

The proof of  lemma \ref{cell distribution lemma} is based on the following simple observation\myr{to pay attention...}.  Given $f$ with $\rr_+^2$-center of mass  at $(\cdot,y_f)$ and variation
$V_f= \int |x_2-y_f|^2 f(x)dx$, we find that its total mass \emph{outside} the strip $\mathcal{S}[y_f,r]:= \{(x_1,x_2)||x_2-y_f|\leq r\}$ with radius $r=R\sqrt{2V_f/M_f}$, does not exceed
\begin{align*}
\int_{|x_2-y_f|>r}f(x)dx=\int_{|x_2-y_f|>r}f(x)\frac{|x_2-y_f|^2}{|x_2-y_f|^2}dx
\leq \frac{{M_f}}{2R^2V_f}\int f(x) |x_2-y_f|^2 dx=\frac{M_f}{2R^2}
\end{align*}
If we can find the $\delta$ such that our target strip $\mathcal{S}_\delta:=\{|x_2|\leq 2\delta\}$ is lying below and \emph{outside} the strip $\mathcal{S}[y_f, r]$, then the total mass in the strip $\mathcal{S}_\delta$ would be smaller than $\frac{1}{2R^2}M_f$. When $n^\ep(x,t)$ takes the role of $f(x)$ with $(y_f,V_f) \mapsto (\y+(t),\V+(t))$, the aim is to bound the strip $\mathcal{S}[\y+(t), r(t)]$ with radius $r(t)=R\sqrt{2\V+(t)/\M+}$ away from a fixed strip $\mathcal{S}_\delta$.
To this end we collect the necessary estimates on $\y+(t), \V+(t)$ and complete the proof of the lemma in section \ref{sec:step1}.

The \emph{second step}, carried in section \ref{sec:step2}, is to prove  the main theorem under a \emph{constrained setup}:  equipped with lemma \ref{cell distribution lemma} we can control the entropy and prove a weaker form of our main theorem for  any  $M<16\pi$ (which is still  larger than the $8\pi$ barrier), under the constraint that  the mass $M$ concentrates far enough from the $x_1$-axis. This is quantified in our next theorem.
We recall that a small parameter $0<\eta\ll1$ was already quantified in lemma \ref{cell distribution lemma}.
\begin{thm}\label{thm 1}
Consider the PKS equation \eqref{KS advection} with background fluid velocity \eqref{vector field b} subject to $H^s(s\geq 2)$ initial data, symmetric about the $x_1$-axis, with mass $M=|n_0|_1< 16\pi$, and bounded second moment $(1+|x|^2)n_0\in L^1(\rr^2)$.
If $\displaystyle R^2=M_+\frac{\y+^2(0)}{2\V+(0)}$ is large enough so that
\begin{equation}\label{constrained setup}
R^2 > \frac{(1+\eta)^2M}{16\pi-M}, \qquad
\end{equation}
then there exists a large enough  $A=A( M, {\x2+(0)}, \V+(0),\eta)$, such that the free energy solution to PKS  \eqref{KS advection},\eqref{vector field b} exists for all time.
\end{thm}
Observe that $R$ is a dimensionless parameter and $R$ being large indicates that most of the mass concentrates away from the `critical' strip, ${\mathcal S}_\delta$ along the $x_1$-axis, namely --- either the center of mass $\y+$ is far enough and/or the mass variation $\V+/M_+ $ is small enough. Either way, we find that for any $M<16\pi$, if $R$ is large enough so that \eqref{constrained setup} holds, then \eqref{thm 1} yields global existence. Although  theorem \ref{thm 1} is not as sharp as the main theorem, its proof  is more illuminating and can be extended easily to prove the main theorem for the `limiting case' of any $M<16\pi$. We therefore include its proof  in section \ref{sec:step2}.\newline
Finally, the \emph{third step} carried in section \ref{sec:step3} presents the proof of the main theorem \ref{thm 2}.\newline
We turn to a detailed discussion of the three steps.

\subsection{Step 1--- control of the cell density distribution}\label{sec:step1}
As  pointed out before, the proof involves the calculation of $\x2+(t)$ and $V_+(t)$, summarized in  the following two lemmas. Here and below, we let $A\lesssim B$ denote  the relation $A\leq CB$ with a constant $C$ which is independent of $\delta$.

\begin{lem}\label{Lem:x2+}
Consider the regularized PKS  equation \eqref{rPKS} with background fluid velocity \eqref{vector field b}. Assume that the initial center of mass $\y+(0)$ is bounded away  from the $x_1$-axes in the sense that \eqref{initial configuration 1} holds. Then \textcolor{blue}{for any  sufficiently small $\delta>0$}, there exists a large enough amplitude of the ambient vector field, $A=\delta^{-2}\M+$, and a constant \textcolor{blue}{$C$ (independent of $\delta$)},  such that $\x2+(0)-C\delta>0$ and the time evolution of \eqref{rPKS},\eqref{vector field b} pushes  the center of mass, $\x2+(t)$,  away from the critical strip, namely
\bel\label{x2+ lower bound}
\x2+(t)\geq \left[{\x2+}(0)-C\delta\right]e^{At}, \quad \textcolor{blue}{A=\frac{M_+}{\delta^2}}.
\eel
\end{lem}

\begin{lem}\label{Lem:V+/M+}
Consider the regularized PKS equation \eqref{rPKS} with background fluid velocity \eqref{vector field b}. Assume that the initial variation around the center of mass $\V+(0)$ is not too large  in the sense that \eqref{initial configuration 1} holds. Then \textcolor{blue}{for any sufficiently small $\delta>0$}, there exists a constant $C=C(V_+(0))$ such that the variation  $\V+(t)$ remains bounded from above,
\bel\label{V_+/M_+ estimate with ep and V_+/M_+_0}
\V+(t)\leq  \left[C\M+\delta+\V+(0)\right]e^{2At}.
\eel
\end{lem}
We note that all the calculations made below should be carried out at the level of the regularized equation \eqref{rPKS}, but for the sake of simplicity, we proceed at the formal level using the weak formulation \eqref{weak formulation for the upper half plane}. We explicitly point when there is a technical subtlety  in the derivation due to difference between the regularized and weak formations.

We begin with the proof of Lemma \ref{Lem:x2+}.
To calculate the dynamics of the center of mass, one needs to formally test the equation with respect to $\varphi(x)=x_2$. To stay away from the critical strip ${\mathcal S}_\delta$, however, we introduce an approximate test function\footnote{To be precise, one should use here the usual  argument of  a further smooth cut-off for large $\varphi$ at $|x_1|+x_2\gg 1$, and recover the result uniformly with respect to that cut-off parameter.} $\varphi \approx x_2$ with a cut-off away from ${\mathcal S}_\delta$
\begin{align*}
\varphi:=\left\{\begin{array}{lll}x_2 &\quad x_2\in (2\delta, \infty),\\
0 &\quad x_2\in (-\infty, \delta),\\
smooth &\quad x_2\in (\delta, 2\delta).\end{array}\right.
\end{align*}
Note that there exists a constant $C_\varphi$ such that
$|\varphi|\leq 2\delta,\quad \forall x_2\leq 2\delta$ and
$|\na\varphi|+\delta |\na^2 \varphi|\leq {C_\varphi}$.
Here and below,  we use $C_\varphi$ to denote $\varphi$-dependent constants \textcolor{blue}{that otherwise are independent of $\delta$}.\newline
Moreover,  replacing $\ds \int_{\rr^2_+} x_2ndx$ with $\ds \int_{\rr^2_+} \varphi ndx$, we lose information on the stripe $\left\{(x_1,x_2)||x_2|\leq 2\delta\right\}$; however, the contribution of this part is small in the sense that:
\bel \label{nx contribution of the small stripe}
\left|\int_{0\leq x_2\leq 2\delta}\varphi n_+ dx-\int_{0\leq x_2\leq 2\delta} x_2 n_+ dx\right|\leq 4M_+\delta.
\eel

Next, one can use $\varphi$ and the weak formulation (\ref{weak formulation for the upper half plane}) to extract information about ${y}_+$:
\begin{align}
\frac{d}{dt}\int_{\rr_+^2}\varphi n_+dx=&\int_{\rr_+^2}\de \varphi n_+dx-\frac{1}{4\pi}\int_{\rr_+^2\times\rr_+^2}\frac{(\na\varphi(x)-\na \varphi(y))\cdot(x-y)}{|x-y|^2}n_+(x,t)n_+(y,t)dxdy\nonumber\\
&+\int_{\rr^2_+}n_+\na c_-\cdot\na \varphi dx+\int_{\rr_+^2}\na \varphi\cdot \bb n_+dx\nonumber\\
=& I+II+III+IV.\label{weak formulation to calculate x2+}
\end{align}
Now we estimate the right hand side of \eqref{weak formulation to calculate x2+} term by term. The first and second terms are relatively easy to control from above,
\begin{subequations}\label{eqs:upper}
\begin{align}
|I|&=\left|\int_{\rr^2_+}\de \varphi n_+dx\right|\leq \frac{C_\varphi M_+}{\delta},\label{I in x2+} \\
|II| &\leq \frac{1}{4\pi}|\na^2\varphi|_\infty\int\int_{\rr^2_+\times\rr^2_+}n_+(x)n_+(y)dxdy
\leq \frac{1}{4\pi}\frac{C_\varphi}{\delta}M_+^2.\label{II in x2+}
\end{align}
Next we upper-bound the third term in \eqref{weak formulation to calculate x2+} in the following way,
\bel
|III|
=\left|\int_{\rr^2_+} \pa_{x_2} c_- n_+\pa_{x_2}\varphi dx\right|
\leq \frac{C_\varphi}{2\pi}\frac{M_+}{\delta}\int_{\rr^2_+}n_+dx
\leq \frac{ C_\varphi}{2\pi}\frac{M_+^2}{\delta}.\label{III in x2+}
\eel
\end{subequations}
This follows from the  pointwise bound $\ds |\pa_{x_2} c_-(x)|\leq \frac{1}{2\pi}\frac{M_+}{|x_2|}$  shown below, and the fact that $supp(\varphi)$ is $\delta$ away from the $x_1$-axis, and hence  $\ds \frac{1}{|x_2|}\leq \frac{1}{\delta}$,
\begin{align}
|\pa_{x_2} c_-(x)|=&\frac{1}{2\pi}\left|\int_{\textcolor{blue}{(\rr^2_+)^c}}\frac{(x-y)_2}{|x-y|^2}n_-(y)dy\right| \nonumber \\
& \leq \frac{1}{2\pi}\frac{1}{|x_2|}\int_{\textcolor{blue}{(\rr^2_+)^c}}\frac{|x_2|}{|x_2|+|y_2|}n_-(y)dy
 \leq   \frac{1}{2\pi}\frac{M_-}{|x_2|}  = \frac{1}{2\pi}\frac{M_+}{|x_2|}. \label{c estimate}
\end{align}

Finally, we need to address additional transport term $IV$ in \eqref{weak formulation to calculate x2+} to compete with the focusing effect. Recall that $\bb=A(-x_1,x_2) $ with $\ds A=\frac{M_+}{\delta^2}$.
First we write $IV$ down explicitly,
\be
IV=\int \bb\cdot n_+ \na \varphi dx=\frac{M_+}{\delta^2}\int_{\rr^2_+} x_2\pa_{x_2} \varphi n_+dx.
\ee
Next we replace the right hand side by $\int \varphi n_+dx$. Due to the fact that $x_2\pa_{x_2}\varphi= x_2=\varphi$ for $x_2>2\delta$, the error introduced in this process originates  from the thin $2\delta$-strip:
\be
\left|\int_{\delta<x_2<2\delta}(x_2\pa_{x_2} \varphi-\varphi)n_+dx\right|\leq \int_{\delta<x_2<2\delta}|x_2\pa_{x_2} \varphi-\varphi|n_+dx\leq C_\varphi\delta M_+,
\ee
and we conclude that
\begin{align}
IV
\geq  \frac{M_+}{\delta^2}\left(\int_{\rr^2_+} \varphi n_+dx- C_\varphi\delta M_+\right)
\geq  \frac{M_+}{\delta^2}\int_{\rr^2_+}\varphi n_+dx-\frac{C_\varphi M_+^2}{\delta}.\label{IV in x2+}
\end{align}

Combining  equation \eqref{weak formulation to calculate x2+} with the upper-bounds \eqref{eqs:upper} and the lower-bound  \eqref{IV in x2+} \textcolor{blue}{while recalling the definition of $A$  \eqref{x2+ lower bound}}, we arrive at
\be\ba
\frac{d}{dt}\int_{\rr^2_+} \varphi n_+ dx\geq
 -C_\varphi\frac{M_+^2}{\delta}+A\int_{\rr^2_+} \varphi n_+ dx,
\ea\ee
which implies that
\bel\label{n varphi lower bound}
\int_{\rr^2_+}\varphi n_+dx\geq \left(\int_{\rr^2_+} \varphi n_0dx- C_\varphi M_+\delta\right)e^{At}.
\eel

Finally, we calculate the center of mass of the upper half plane using the lower bound \eqref{n varphi lower bound} and the error control (\ref{nx contribution of the small stripe}):
\be\ba
\x2+(t)
=&\frac{1}{\M+}\left({\int_{0\leq x_2\leq 2\delta} x_2n_+dx}-{\int_{0\leq x_2\leq 2\delta} \varphi n_+dx}+{\int_{\rr^2_+} \varphi n_+dx}\right)\\
\geq &-\frac{1}{\M+}\left|{\int_{0\leq x_2\leq 2\delta} x_2n_+dx}-{\int_{0\leq x_2\leq 2\delta} \varphi n_+dx}\right|+\frac{1}{\M+}{\int_{\rr^2_+} \varphi n_+dx}\\
\geq &-4\delta+\frac{1}{M_+}\left(\int_{\rr^2_+} \varphi n_0dx-C_\varphi M_+\delta\right)e^{At}\\
\geq &\left(\x2+(0)-C_\varphi\delta\right)e^{At}.
\ea\ee
This completes the proof of lemma \ref{Lem:x2+}.$\Box$

\begin{rmk}\label{rmk on regularization}
The only difference in  estimating  the regularized solutions \eqref{rPKS} vs. the formal calculation we have done above is in terms II and III. In the calculation for the \eqref{rPKS}, we will need the estimate
\be
|\na K^\ep (z)|\leq \frac{1}{2\pi |z|},\quad \forall z\in \rr^2.
\ee
Here we show how to get a similar estimate for term $II$ in \eqref{weak formulation to calculate x2+} for the regularized equation \eqref{rPKS}:
\be\ba
II=&\left|\iint_{\rr^2_+\times\rr^2_+} \na \varphi(x) \na_x [K^{\ep}(|x-y|)]n^\ep(y)n^\ep(x)dxdy\right|\\
=&\frac{1}{2}\left|\iint_{\rr^2_+\times\rr^2_+}\frac{(\na \varphi(x)-\na \varphi(y))\cdot(x-y)}{|x-y|}\na K^{\ep}(|x-y|)n^\ep(x)n^\ep(y)dxdy\right|\\
\leq &\frac{1}{4\pi}\left|\iint_{\rr^2_+\times\rr^2_+}\frac{|\na \varphi(x)-\na \varphi(y)|}{|x-y|}n^\ep(x)n^\ep(y)dxdy\right| \leq \frac{1}{4\pi}|\na^2\varphi|_\infty M_+^2
\leq \frac{1}{4\pi}\frac{C_\varphi}{\delta}M_+^2.
\ea\ee
The treatment of term $III$ is similar to the one we gave above.
\end{rmk}

Next we address the proof of Lemma \ref{Lem:V+/M+}.
The main goal is to calculate time evolution of the variation
\be
V_+(t):=\int_{\rr^2_+} |x_2-\x2+(t)|^2n(x,t)dx.
\ee
We again use $C$ to denote constants which may change from line to line but are independent of  $\delta$.

The first obstacle is that we cannot choose $|x_2-\x2+|^2$ as a test function due to the fact that $\x2+(t)$ depends on the solution. However, by the definition of $\y+$ we can expand the $\V+$-integrand, ending up with the usual
\begin{eqnarray}\label{V+ relation}
\V+(t)=\int_{\rr^2_+} |x_2|^2n_+(x,t)dx- \M+ {\x2+^2}(t).
\end{eqnarray}

Since we already know $\x2+$, it is enough to calculate the $\int_{\rr^2_+} |x_2|^2n(x,t)dx$. For simplicity, we plug $|x_2|^2$ inside the weak formulation \eqref{weak formulation for the upper half plane} and \eqref{V+ relation} to get the time evolution of $V_+$. Of course, what one really does is to use a test function to approximate the $|x_2|^2$. Furthermore, when we use the weak formulation, we formally integrated by part twice, but since the value and the first derivative of the function $|x_2|^2$ are zero on the boundary, we will not create extra dangerous boundary term.

First combining (\ref{weak formulation for the upper half plane}) and (\ref{V+ relation}) yields
\begin{align}
\frac{d}{dt}V_+=&\frac{d}{dt}\int_{\rr^2_+} |x_2|^2n_+(x,t)dx-\M+\frac{d}{dt} {\x2+^2}(t)\nonumber\\
=&\int_{\rr_+^2}\de \varphi n_+dx-\frac{1}{4\pi}\int_{\rr_+^2\times\rr_+^2}\frac{(\na\varphi(x)-\na \varphi(y))\cdot(x-y)}{|x-y|^2}n_+(x,t)n_+(y,t)dxdy\nonumber\\
&+\int_{\rr^2_+}\na c_-n_+\na \varphi dx+\int_{\rr_+^2}\na \varphi\cdot \bb n_+dx-\frac{d}{dt}\big( M_+ (\x2+)^2\big)\nonumber\\
=&I+II+III+IV-\M+\frac{d}{dt}{\x2+^2}(t)\label{weak formulation for V+}
\end{align}
Next we estimate every term on the right hand side of \eqref{weak formulation for V+}. The first two terms are estimated as follows:
\bel
|I|=&\ds \left|\int_{\rr_+^2}\de( x_2^2) n_+dx\right|=2M_+,\label{I in V+}
\eel
and
\begin{align}
|II|=&\left|-\frac{1}{4\pi}\int_{\rr_+^2\times\rr_+^2}\frac{(\na(x_2^2)-\na (y_2^2))\cdot(x-y)}{|x-y|^2}n_+(x,t)n_+(y,t)dxdy\right|\nonumber\\
\leq &\frac{1}{4\pi}\int_{\rr_+^2\times\rr_+^2}2 n_+(x)n_+(y)dxdy\leq \frac{1}{2\pi} M_+^2.\label{II in V+}
\end{align}
Now for the third term in \eqref{weak formulation for V+}, we use the previous upper-bound   (\ref{c estimate}), $|\pa_{x_2} c_-(x)|\leq \frac{1}{2\pi}\frac{M_+}{|x_2|}$, which yields
\bel
|III|
=\left|\int_{\rr^2_+}\pa_{x_2}c_-n_+2x_2dx\right|\leq \frac{1}{2\pi}\int_{\rr^2_+}2|x_2|\frac{M_+}{|x_2|}n_+dx\leq \frac{1}{\pi} M_+^2.\label{III in V+}
\eel
Note that for the term $II$ and $III$, we only estimate them formally above, one can prove the estimates explicitly using the same techniques as the one in Remark \ref{rmk on regularization}.
For the $IV$ term in \eqref{weak formulation for V+}, we use the (\ref{V+ relation}) again to obtain
\bel
|IV|=\left|\int_{\rr_+^2}\na x_2^2 \cdot \bb n_+dx\right|=2A\left|\int_{\rr^2_+}x_2^2 n_+dx\right|=2A(V_++M_+\x2+^2)\label{IV in V+}
\eel

Collecting equation \eqref{weak formulation for V+} and all the estimates \eqref{I in V+}, \eqref{II in V+}, \eqref{III in V+} and \eqref{IV in V+} above, we have the following differential inequality,
\be
\frac{d}{dt}\left(\frac{1}{\M+}\V+(t)+ {\x2+^2}(t)\right)\leq C(1+ M_+)+2A\left(\frac{1}{\M+}\V+(t)+{\x2+^2}(t)\right),
\ee
which yields
\begin{equation}\label{V_+/M_++(x^*)^2}
\frac{1}{\M+}\V+(t)+{\x2+^2}(t)\leq C\left(1+\frac{1}{\M+}\right)\delta^2 e^{2At}+\frac{1}{\M+}\V+(t) e^{2At}+{\x2+^2}(0) e^{2At}.
\end{equation}
Combining Lemma \ref{Lem:x2+} with (\ref{V_+/M_++(x^*)^2}) yields
\begin{align*}
\frac{1}{\M+}\V+(t)+\left[\left({\x2+}(0)-C \delta\right)e^{At}\right]^2
& \leq \frac{1}{\M+}\V+(t)+{\x2+^2}(t)\\
& \leq  C\left(1+\frac{1}{\M+}\right)\delta^2 e^{2At}+\frac{1}{\M+}\V+(0) e^{2At}+{\x2+^2}(0) e^{2At}.
\end{align*}
By collecting similar terms, we finally have
\bel
\frac{1}{\M+}\V+(t) \leq  \left[2C\delta{\x2+}(0)+C\left(1+\frac{1}{M_+}\right)\delta^2+\frac{1}{\M+}\V+(0)\right]e^{2At}.
\eel
which completes the proof of Lemma \ref{Lem:V+/M+}. $\Box$

\bigskip
Equipped with the estimate on $\V+(t)$, we can now conclude the proof of  Lemma \ref{cell distribution lemma}.

\begin{proof} (Lemma \ref{cell distribution lemma})
Once $0<\eta\ll 1$ was fixed, we can clearly choose a small enough $\delta$ such that by \eqref{V_+/M_+ estimate with ep and V_+/M_+_0}${}_\delta$,  there holds
\bel\label{V_+/M_+ upper bound}
\V+(t)\leq{(1+\eta)}\V+(0) e^{2At}.
\eel
Now recalling  that $R=\x2+(0)\sqrt{\frac{\M+}{2\V+(0)}} > 1$, then we can use (\ref{x2+ lower bound}), (\ref{V_+/M_+ upper bound}) and further choose $\delta$ small enough to get:
\begin{align*}
\x2+(t)-\frac{R}{1+\eta}\sqrt{\frac{2\V+(t)}{\M+}}\geq&\left[\x2+(0)-C \delta-\frac{R}{\sqrt{1+\eta}}\sqrt{\frac{2\V+(0)}{\M+}}\right] e^{At}\\
=&\left[\bigg(1-\frac{1}{\sqrt{1+\eta}}\bigg)\x2+(0)-C \delta\right] e^{At}\\
\geq& 2\delta e^{At}\geq 2\delta.
\end{align*}
Thus, the `thin' $\delta$-strip along the $x_1$-axis, $ {\mathcal S}_\delta:=\{(x_1,x_2)|0\leq x_2\leq 2\delta\}$, lies \emph{outside} the strip centered around $\y+(t)$, uniformly in time,
\[
{\mathcal S}_\delta \subset \{(x_1,x_2)\ |\ |x_2-\x2+(t)|>R_\eta(t)\}, \qquad R_\eta(t):=\frac{R}{1+\eta}\sqrt{\frac{2\V+(t)}{\M+}}.
\]
It follows  that thanks to our choice of $\delta$, the mass inside the $\delta$-strip $ {\mathcal S}_\delta$ does not exceed
\be\ba
\int_{{\mathcal S}_\delta}n_+(x,t)dx\leq&\int_{\rr^2_+\cap \left\{|x_2-\x2+|>R_\eta\right\}}n_+(x,t)dx \leq  \int_{\rr^2_+\cap \left\{|x_2-\x2+|>R_\eta\right\}}n_+(x,t)\frac{|x_2-\x2+|^2}{|x_2-\x2+|^2}dx\\
 \leq & \frac{(1+\eta)^2}{R^2 2\V+/\M+}\int_{\rr^2_+}n_+(x,t)|x_2-\x2+|^2dx = \frac{(1+\eta)^2}{R^22\V+/\M+}\V+\\
  \leq & \frac{(1+\eta)^2}{2R^2}M_+.
\ea\ee
 By symmetry,  the mass inside the symmetric $\delta$-strip, $\ds \{(x_1,x_2)||x_2|\leq 2\delta\}$ is smaller than $\ds \frac{(1+\eta)^2}{2R^2}2\M+=\frac{(1+\eta)^2}{2R^2}M$, uniformly in time, which completes the proof of Lemma \ref{cell distribution lemma}.
\end{proof}

\begin{rmk}
One can do a similar computation to get the evolution for the higher moment estimates $\ds\int_{\rr^2_+}n(x,t)|x|^{2k}dx$, and derive similar results to Lemma \ref{cell distribution lemma}.
\end{rmk}

\subsection{Step 2 --- proof of the main theorem under the constrained setup
\eqref{constrained setup}}\label{sec:step2}
With the Lemma \ref{cell distribution lemma} at our disposal, we can now turn to  the proof of theorem \ref{thm 1} along the lines of \cite{BCM2008}. Note that the actual calculation are to be carried out with the regularized solutions $n^\ep$ of \eqref{rPKS}, though  for the sake of simplicity, we only do the formal calculation on $n(x)=n(\cdot,t)$.

The key is to use the logarithmic Hardy-Littlewood-Sobolev inequality to get a bound on the entropy $S[n]$.
\begin{thm}[Logarithmic Hardy-Littlewood-Sobolev Inequality]\cite{CL92} Let $f$ be a nonnegative function in $L^1(\rr^2)$ such that $f\log f$ and $f\log(1+|x|^2)$ belong to $L^1(\rr^2)$. If $\int_{\rr^2}fdx=M$, then
\bel\label{log HLS}
\int_{\rr^2}f\log f dx+\frac{2}{M}\iint_{\rr^2\times\rr^2} f(x)f(y)\log|x-y|dxdy\geq -C(M)
\eel
with $C(M):=M(1+\log \pi-\log M)$.
\end{thm}
\begin{rmk} It is pointed out in \cite{BCM2008} that by multiplying $f$ by indicator functions, one can prove that the inequality \eqref{log HLS} remains true with $\rr^2$ replaced by any bounded domains $\mathcal{D}\subset \rr^2$.
\end{rmk}

The idea of the proof goes as follows. By observing that the mass in the upper half plane and lower half plane are subcritical ($\displaystyle{|n_\pm|_1<8\pi}$), we plan to use the logarithmic Hardy-Littlewood-Sobolev inequality on these sub-domains to get uniform bound on the entropy. However, without extra information concerning the cell density distribution, naive application of logarithmic Hardy-Littlewood-Sobolev inequality fails.  For this approach to work, the density distribution constraint required is that the cells in the upper and lower half plane are well-separated by a 'cell clear strip' in which the total number of cells is sufficiently small. The strip is constructed through applying Lemma \ref{cell distribution lemma}. Combining the logarithmic Hardy-Littlewood-Sobolev inequality and the cell seperation constraint, we can use a 'total entropy reconstruction' trick introduced in \cite{BCM2008} to obtain entropy bound.
Now let's start the whole proof.
\begin{proof}[\textcolor{blue}{Proof of Theorem \ref{thm 1}}]
According to propositions \ref{pro 2.1}, \ref{pro 1.2}, the life-span of free-energy solution is determined by the finite bound on the entropy $S[n](t)$.
 To this end, we decompose the free energy into three parts,
\bel\label{T1T2}
\ba
E[n]\equiv &\left(1-\frac{K}{8\pi}\right)\int_{\rr^2} n\log ndx  \\
& + \frac{1}{8\pi}\bigg(K\int_{\rr^2} n\log ndx+2\iint_{\rr^2\times\rr^2}n(x)n(y)\log |x-y|dxdy\bigg)-\int_{\rr^2} \textcolor{blue}{H}ndx\\
=:&\left(1-\frac{K}{8\pi}\right)S[n]+\T_1-\T_2.
\ea
\eel
Recall that  since the free energy $E[n](t)$ in \eqref{E}   is decreasing,
\be\label{eq:finally}
\left(1-\frac{K}{8\pi}\right)S[n] \equiv E[n] -\T_1+\T_2 \leq  E[n_0] -\T_1+\T_2,
\ee
then  the desired entropy bound and hence a finite entropy solution  follow provided we show the existence of a constant  $K<8\pi$, for which $\T_1=\T_1(K)$ is bounded from below and  $\T_2$ is bounded from above. Our main task is  estimating $\T_1(K)$ from below, which is further decomposed into three terms
\begin{align*}
\T_1(K) & = K\int_{\rr^2} n\log^+ ndx+2\iint_{\rr^2\times\rr^2}n(x)n(y)\log |x-y|dxdy - K\int_{\rr^2} n\log^- ndx \\
&  =: K \T_{11}+\T_{12}+ K\T_{13}.
\end{align*}
To estimate the various terms, we first  construct the 'cell clear strip', 
$\Gc$ shown in figure \ref{fig:1},
\begin{equation}
\G+:=\{x_2\, | \, x_2>2\delta\},\quad\Gm:=\{x_2\, |\, x_2<-2\delta\},\quad\Gc:=\{x_2\ | \ |x_2|\leq 2\delta\}.
\end{equation}
Thus, region $\G+$ contains points in the upper half plane which are $2 \delta$ away from the $x_1$ axis, whereas region $\Gm$ contains points in the lower half plane with the same property. Region $\Gc$ is a closed $2\delta$-neighborhood of the $x_1$-axis. The $\delta$ neighborhood of the $\G+,\Gm$ region is denoted as follows:
\begin{equation}
\G+^{(\delta)}=\{x_2 \ | \ x_2>\delta\},\quad\Gm^{(\delta)}=\{x_2\, | \, x_2<-\delta\}.
\end{equation}
In the sequel, we will further decompose $\Gc$ into subdomains:
\begin{equation}
\S+=\{x_2\,|\,\delta<x_2\leq 2\delta\},\quad \Sm=\{x_2\,|-\,\delta> x_2\geq-2\delta\}, \quad \Sc=\{x_2\ | \ |x_2|\leq \delta\}.
\end{equation}

\begin{figure}\label{fig:1}
  \centering
  \includegraphics[width=8cm]{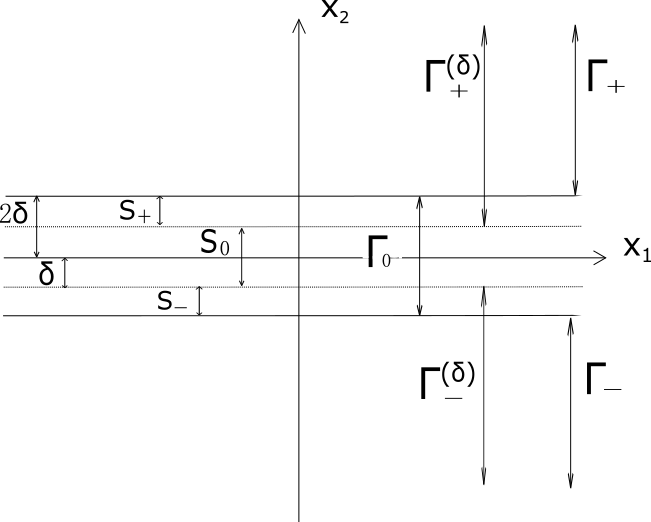}\\
  \caption{Regions $\G+,\Gm,\Gc$}
\end{figure}

\noindent
{\bf Estimating $K\T_{11}+\T_{12}$ from below}.
We fix $K$ as
\bel\label{K}
K:=\max\bigg\{\int_{\G+^{(\delta)}}ndx+\int_{\Gc}ndx,\int_{\Gm^{(\delta)}}ndx+\int_{\Gc}ndx\bigg\},
\eel
and claim that $K$ is admissible. Indeed, 
 lemma \ref{cell distribution lemma} with our choice of $R^2$ assumed large enough so that \eqref{constrained setup} holds implies  that the total mass inside $\Gc$ does not exceed
\begin{align}\label{Gamma3mass}
\int_{\Gc}ndx=\int_{\{x_2\,|\, |x_2|\leq 2\delta\}}ndx\leq \frac{(1+\eta)^2}{2R^2}M \leq 8\pi-\frac{M}{2}
\end{align}
Therefore, the $\Gc$ strip is the 'cell clear strip'; moreover, since the  upper-half plane mass equals $\frac{M}{2}$, the admissibility of $K$ follows, 
$\ds K\leq \frac{M}{2}+\frac{(1+\eta)^2}{2R^2}M < 8\pi$.\newline

The motivation for this choice of $K$ comes from the straightforward bound
\bel\label{eq:key}
\ba
\int_{\G+^{(\delta)}}n &dx\int_{\G+^{(\delta)}} n\log^+ n dx+\int_{\Gm^{(\delta)}}ndx\int_{\Gm^{(\delta)}} n\log^+ n dx+\int_{\Gc}ndx\int_{\Gc} n\log^+ n dx\\
& \leq\int_{\G+^{(\delta)}}ndx \int_{\rr^2_+}n\log ^+n dx + \int_{\Gm^{(\delta)}}ndx \int_{(\rr^2_+)^c}n\log ^+n dx+\int_{\Gc} ndx\int_{\rr^2} n\log^+ ndx\\
& \leq{\max\left\{\int_{\G+^{(\delta)}}ndx+\int_{\Gc} ndx,\int_{\Gm^{(\delta)}}ndx +\int_{\Gc} ndx\right\}}\cdot \int_{\rr^2}n\log ^+n dx =:  K\T_{11}.
\ea
\eel
Indeed, we proceed along the lines of \cite{BCM2008},  appealing to the Logarithmic Hardy-Littlewood-Sobolev inequality in the three regions $\textcolor{blue}{\G+^{(\delta)}},\Gm^{(\delta)},\Gc$ , obtaining
\begin{align*}
&\int_{\G+^{(\delta)}}n(x)dx\int_{\G+^{(\delta)}}n\textcolor{blue}{\log^+ n}dx+2\iint_{\G+^{(\delta)}\times \G+^{(\delta)}}n(x)n(y)\log |x-y|dxdy\geq C,\\
&\int_{\Gm^{(\delta)}}n(x)dx\int_{\Gm^{(\delta)}}n\textcolor{blue}{\log^+ n}dx+2\iint_{\Gm^{(\delta)}\times \Gm^{(\delta)}}n(x)n(y)\log |x-y|dxdy\geq C,\\
&\int_{\Gc}n(x)dx\int_{\Gc}n\textcolor{blue}{\log^+ n}dx+2\iint_{\Gc\times \Gc}n(x)n(y)\log |x-y|dxdy\geq C.
\end{align*}
We now sum these three inequalities: by \eqref{eq:key}, the sum of their first three terms  does not exceed $K\T_{11}$;  bookkeeping the overlap of the three domains $\G+^{(\delta)}\times \G+^{(\delta)}, \Gm^{(\delta)}\times \Gm^{(\delta)}$
and $\Gc\times \Gc$, consult figure \ref{Fig:Region_R} we find
the sum of the  second three terms is $\T_{12}$ modulo the correction $\T_{121}-\T_{122}$ below,
\bel\label{I1 I2 -I3 I4}
\ba
K\T_{11}+\T_{12} \geq &-C +4\iint_{((\G+^{(\delta)})^c\times \G+)\cup (\Gm\times(\S+\cup \Sc))}n(x)n(y)\log|x-y|dxdy\\
&-2\iint_{(\S+\times \S+)\cup(\Sm\times \Sm)}n(x)n(y)\log|x-y|dxdy\\
=:&-C+\T_{121}-\T_{122}.
\ea
\eel
Next applying the fact that $|x-y|\geq \delta$ for all $(x,y)$ in the integral domain of $\T_{121}$, we estimate the $\T_{122}$ and $\T_{122}$,
\begin{align} \T_{121}\geq &4M^2\log \delta,\nonumber\\
\T_{122}\leq &2\iint_{(\S+\times \S+)\cup(\Sm\times \Sm)}n(x)n(y)\log^
+|x-y|dxdy\nonumber\\
\leq&C\iint_{(\S+\times \S+)\cup(\Sm\times \Sm)}n(x)n(y)(1+|x|^2+|y|^2)dxdy\leq C(M^2+2M\int n(x) |x|^2 dx).\label{I3 I4}
\end{align}
Combining \eqref{I3 I4} with \eqref{I1 I2 -I3 I4} yields
$K\T_{11}+\T_{12} \geq 4M^2\log\delta-C(1+M^2+M\int n|x|^2dx)$.

\medskip\noindent
{\bf Estimating $\T_{13}$ from below}.
We recall the well-known upper bound on the negative part of the entropy, \cite{BDP2006},\cite{BCM2008}, stating that for $f$ positive function, 
\bel\label{negative part of entropy}
\int_{\rr^2}f\log^- fdx\leq \frac{1}{2}\int_{\rr^2}|x|^2fdx+\log(2\pi)\int_{\rr^2}fdx+\frac{1}{e}\leq C\left(1+M+\int n |x|^2 dx\right).
\eel
Combining \eqref{negative part of entropy} with our previous lower-bound of $K\T_{11}+\T_{12}$  yields
\bel\label{T1}\ba
\T_1(K)=K\T_{11}+\T_{12}+K\T_{13} \geq  4M^2\log\delta  -C(M+K)\left(1+M+\int n|x|^2dx\right).
\ea\eel
It remains to upper-bound  the $\T_2$ term in \eqref{T1T2}. Since  $\displaystyle{\bigg|\int Hndx\bigg|\leq A\int n|x|^2dx}$, it is therefore suffices  to show that the second moment of $n$ is bounded for any finite time. Indeed, the time evolution of the second moment can be estimated as follows,
\begin{align*}
\frac{d}{dt}\int n|x|^2dx
\leq4AM+4A\int H n(x)dx
\leq4AM +\frac{A^2}{2}\int  n|x|^2dx,
\end{align*}
and Gronwall inequality yields  the finite bound $\ds |\T_2| \leq \int n(\cdot,t) |x|^2dx\leq C(A,t)<\infty$.\newline
Finally, equipped with the lower bound on $\T_1(K)$ with $K<8\pi$ and with the upper bound on $\T_2$  we revisit \eqref{eq:finally} to conclude that there exists a constant $C=C(M,A,T)< \infty$ such that the entropy $S[n(\cdot,t)]$ is uniformly bounded (independent of $\ep$) for any finite time interval $t\in [0,T]$,
\begin{align}
S[n](t)\leq\frac{1}{(1-\frac{K}{8\pi})}\bigg(E[n_0]+C-\frac{1}{2\pi}M^2\log\delta\bigg),\quad\forall t\in[0,T].\label{entropy bound}
\end{align}
 Now by the Propositions \ref{pro 2.1}, \ref{pro 1.2}, we have that the free energy solution exists on any time interval $[0,T],\quad\forall T<\infty$.
\end{proof}

\subsection{Step 3 --- proof of the main theorem}\label{sec:step3}

In the proof of Theorem \ref{thm 1}, we see that the cell population is separated by a 'cell clear zone' near the $x_1$ axis. Since total mass in the "cell clear zone" is small, we can heuristically treat the total cell population as a union of two subgroups with subcritical mass ($<8\pi$). However, since we lack sufficiently good control over the total number of cells near the $x_1$ axis, we cannot use this idea to prove the optimal result as stated in Theorem \ref{thm 2}. The idea of proving Theorem \ref{thm 2} is that instead of considering the total cell population as the union of two subgroups separated by one fixed 'cell clear zone', we treat it as the union of three subgroups with subcritical mass, namely, the cells in the upper half plane, the lower half plane and the neighborhood of the $x_1$ axis, respectively. These three subgroups of cells are separated by two 'cell clear zones' varying in time.

The main difficulty in the proof is setting up the three new regions such that:

1. mass inside each region is smaller than $8\pi$;

2. the total mass of cells near their boundaries is well-controlled.

Once the construction is completed, the remaining steps will be similar to step 2.

\begin{proof}[Proof of Theorem \ref{thm 2}]
We start by constructing the three regions. First we note that the Lemma \ref{cell distribution lemma} implies that there exists $\delta>0$ such that the following estimate is satisfied for a fixed $R>1$ and $\eta$ chosen small enough:
\bel\label{mass in the strip M+}
\int_{|x_2|\leq 2\delta}n dx\leq \frac{(1+\eta)^2}{\myr{2}R^2} \int_{\rr^2}ndx\leq \frac{1}{2}M,\quad\forall t>0.
\eel
Now the region $L=\{(x_1,x_2)||x_2|\leq 2\delta\}$ have  total mass less than $\frac{1}{2}M=M_+<8\pi$ for all time.

Secondly, we subdivide the region $L$ into $J$ pieces:
\be\ba
L={\displaystyle \mathop{\cup}_{1}^{J}} L^{i}, \qquad 
L^i:=\left\{(x_1,x_2)\ \big|\ \frac{2\delta}{J}(i)> |x_2|\geq\frac{2\delta}{J}(i-1)\right\}.
\ea\ee
Here $J=J(M)\geq 10$, to be determined later, depends on $M$. By the pigeon hole principle, there is at least three strips $L^i$ such that
\be
\int_{L^i}n(x)dx\leq \frac{2}{J}M_+.
\ee
Suppose there are only two strips with mass smaller than $\frac{2}{J}M_+$, then total mass in $L$ will be bigger than $(J-2)\frac{2}{J} M_+>M_+$, a contradiction. Now we pick from these three strips the one which is neither $L^1$ nor $L^J$. As a result, this strip $L^i$ does not touch the $x_1$ axis nor the boundary of $L$. We denote this $i$ by $i^*$. The $L^{i^*}$ is the 'cell clear zone'. Notice that here $i^*=i^*(n,t)$ depends on time.

Finally, we use this $i^*$ to define the regions. First we define the three regions, each of which has total mass smaller than $8\pi$:
\[
\G+=\bigg\{x_2\geq \frac{2\delta}{J}i^*\bigg\},\quad \Gm=\bigg\{x_2\leq -\frac{2\delta}{J}i^*\bigg\},\quad\Gc=\bigg\{|x_2|\leq \frac{2\delta}{J}(i^*-1)\bigg\}.
\]
Next we set
\[
\rho=\frac{2\delta}{3J},
\]
and define the $\rho$ neighborhood of the above three regions:
\begin{equation}
\G+^{(\rho)}=\left\{x_2> \frac{2\delta}{J}(i^*-\frac{1}{3})\right\}\quad\Gm^{(\rho)}=\left\{x_2< -\frac{2\delta}{J}(i^*-\frac{1}{3})\right\},\quad \Gc^{(\rho)}=\left\{|x_2|< \frac{2\delta}{J}(i^*-\frac{2}{3})\right\}.
\end{equation}
Now we define the complement $\Gw$ of the above three regions $\Gamma_i^{(\rho)},i=+,-,0$:
\begin{align*}
\Gw=\bigg\{\frac{2\delta}{J}\left(i^*-\frac{2}{3}\right)\leq|x_2|\leq\frac{2\delta}{J}\left(i^*-\frac{1}{3}\right)\bigg\}=\Gamma_{\omega+}\cup \Gamma_{\omega-}, \quad 
\Gamma_{\omega\pm}=\Gw\cap \rr^2_{\pm}.
\end{align*}
Now we define the complement $\Gw^{(\rho)}$ of $\cup_{i\in\{\pm,0\}}\Gamma_i$ and decompose it into subdomains:
\begin{align*}
\Gw^{(\rho)}=&\big(\cup_{i\in\{\pm,0\}}\Gamma_i\big)^c=\big(\Gamma_{\omega+}^{(\rho)}\big)\cup\big(\Gamma_{\omega-}^{(\rho)}\big),\qquad \Gamma_{\omega\pm}^{(\rho)}=\Gw^{(\rho)}\cap\rr^2_\pm,\\
\Gw^{(\rho)}=&\Gw\cup \S+\cup \Sm\cup \Sc, \qquad \qquad  \qquad \qquad
\S+=\G+^{(\rho)}\backslash \G+,\quad \Sm=\Gm^{(\rho)}\backslash \Gm,\quad \Sc=\Gc^{(\rho)}\backslash \Gc.
\end{align*}

\begin{figure}
  \centering
  \includegraphics[width=7cm]{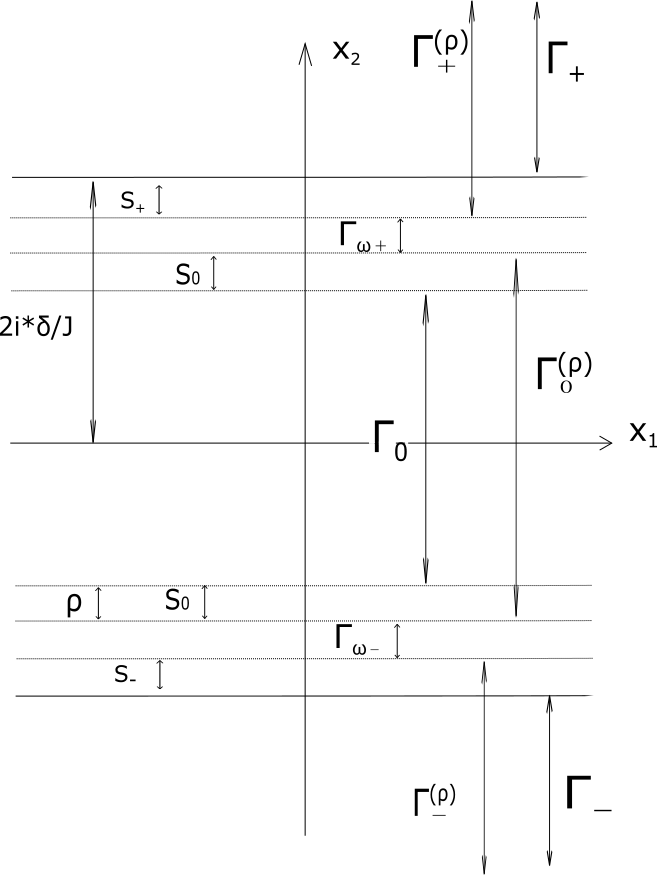}\\
  \caption{Regions $\G+,\Gm,\Gc$ and $\Gw$ in the proof of the main theorem}\label{Fig:G+mc}
\end{figure}

\begin{rmk}
It is important to notice that the regions we are constructing \textcolor{blue}{are} changing with respect to the given time $t$. Therefore, by doing the argument below, we can only show that the entropy is bounded at time $t$, but since $t$ is an arbitrary finite time, we have the bound on entropy for $\forall t\in[0,T],\forall T<\infty$.
\end{rmk}

We start estimating the entropy. By the free energy dissipation, we obtain
\begin{align}
E[n_0]\geq &\bigg(1-\frac{K}{8\pi}\bigg)\int_{\rr^2} n\log ndx\nonumber \\
& +\frac{1}{8\pi}\bigg(K\int_{\rr^2} n\textcolor{blue}{\log^+} ndx
+2\iint_{\rr^2\times\rr^2}n(x)n(y)\log |x-y|dxdy\bigg)\nonumber \\
& -\left(\frac{K}{8\pi}\int_{\rr^2} n\log^- ndx+\int_{\rr^2} Hndx\right)\nonumber\\
=:&\left(1-\frac{K}{8\pi}\right)S[n(T)]+\T_1(K)-\T_2(K)\label{T1T2_2}.
\end{align}
To derive entropy bound, we need to estimate $\T_1$ from below for $K<8\pi$ and estimate $\T_2$ from above. We start by estimating $\T_1$. Combining the definition of $L^{i^*}$ and \eqref{mass in the strip M+} yields
\begin{align}
\int_{\G+^{(\rho)}}ndx\leq &M_+=\frac{1}{2}M<8\pi,\\
\int_{\Gm^{(\rho)}}ndx\leq &M_+=\frac{1}{2}M<8\pi,\\
\int_{\Gc^{(\rho)}}ndx\leq &M_+=\frac{1}{2}M<8\pi,\\
\int_{\Gw^{(\rho)}}ndx\leq &\int_{L^{i^*}}ndx\leq \frac{2}{J}(M_+).\label{Ga 4 mass}
\end{align}
Now by the log-Hardy-Littlewood-Sobolev inequality \eqref{log HLS}, we have that
\begin{align*}
\int_{\Gamma_i^{(\rho)}}n(x)dx\int_{\Gamma_i^{(\rho)}}n(x)\log^+ n(x)dx+2\iint_{\Gamma_i^{(\rho)}\times\Gamma_i^{(\rho)}}n(x)n(y)\log |x-y|dxdy
&\geq -C,\quad i=+,-,0,\\
\int_{\Gamma_{\omega\pm}^{(\rho)}}n(x)dx\int_{\Gamma_{\omega\pm}^{(\rho)}}n(x)\log^+ n(x)dx+2\iint_{\Gamma_{\omega\pm}^{(\rho)}\times\Gamma_{\omega\pm}^{(\rho)}}n(x)n(y)\log |x-y|dxdy
& \geq -C.
\end{align*}
Same as in subsection 3.3, we recall the definition of $\T_1$ \eqref{T1T2_2}, and use the estimates above to \textcolor{blue}{reconstruct} the entropy and the potential on the whole $\rr^2$ as follows:
\begin{align}
-C\leq&\left(K\int_{\rr^2}n(x)\log^+n(x)dx+2\iint_{\rr^2\times \rr^2}n(x)n(y)\log|x-y|dxdy\right)\nonumber\\
&-2\iint_{\Omega}n(x)n(y)\log|x-y|dxdy\nonumber\\
&+2\iint_{(\S+\times \S+)\cup(\Sm\times \Sm)\cup(\Sc^+\times \Sc^+)\cup (\Sc^-\times \Sc^-)}n(x)n(y)\log|x-y|dxdy\nonumber\\
=:&8\pi\T_1-\T_3+\T_4.\label{I1 I2 -I3 I4 2}
\end{align}
The region $\Omega$\footnote{ Region $\Omega$ is the union of the following nine regions, $(1)-(9)$:
\be\ba
\qquad 1)&\G+\times (\G+^{(\rho)})^c,\quad 2)\S+\times (\G+\cup \Gamma_{\omega+}^{(\rho)})^c,\quad 3)\Gamma_{\omega+}\times(\Gamma_{\omega+}^{(\rho)})^c,\quad
4)\Sc^+\times (\Gc^{(\rho)}\cup \Gamma_{\omega+}^{(\rho)})^c,\quad 5)\Gc\times(\Gc^{(\rho)})^c,\\ 
\qquad 6)&\Sc^-\times(\Gc^{(\rho)}\cup \Gamma_{\omega-}^{(\rho)})^c,\quad
7)\Gamma_{\omega-}\times (\Gamma_{\omega-}^{(\rho)})^c,\quad 8)\Sm\times(\Gm^{(\rho)}\cup\Gamma_{\omega-}^{(\rho)})^c,\quad 9)\Gm\times(\Gm^{(\rho)})^c.
\ea\ee}
 and the integral domain of $\T_4$ are indicated in Figure \ref{Fig:Region_R}.
The $K$ in \eqref{I1 I2 -I3 I4 2} can be estimated using \eqref{Ga 4 mass} as follows
\bel\label{K 2}
K:=M_++\int_{\Gw^{(\rho)}}ndx\leq \bigg(1+\frac{2}{J}\bigg)M_+.
\eel
By the assumption $M_+<8\pi$, we can make $J$ big such that $K<8\pi$. This is where we choose the $J=J(M)$.
\begin{figure}
  \centering
  \includegraphics[width=7cm]{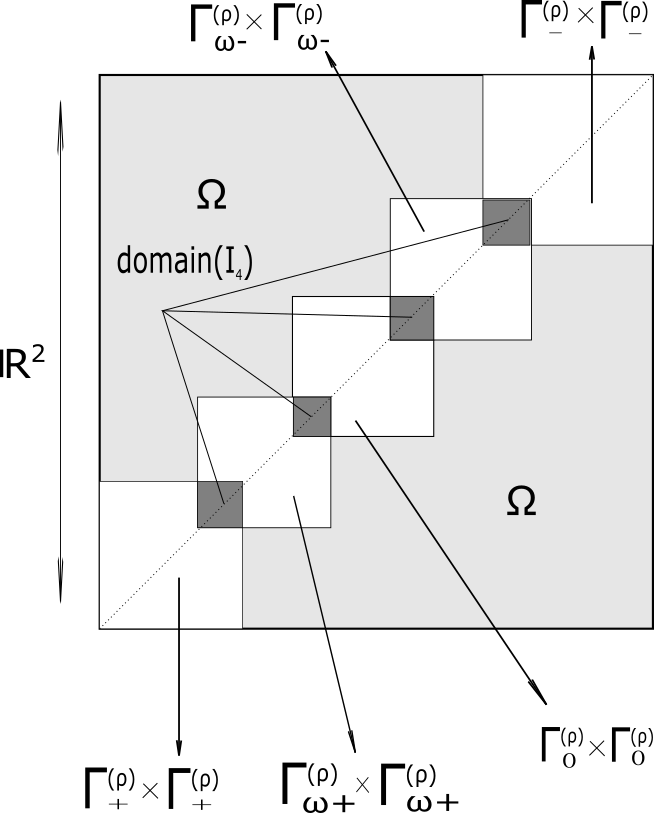}\\
  \caption{Region $\Omega$ in $\rr^2\times \rr^2$}\label{Fig:Region_R}
\end{figure}
Applying the fact that $|x-y|\geq \frac{2\delta}{3J}$, $\forall (x,y)\in \Omega$, the $\mathcal{I}_3$ and $\mathcal{I}_4$ terms in \eqref{I1 I2 -I3 I4 2} can be estimated as follows:
\begin{align}
\mathcal{I}_3\geq &CM^2\log\frac{2\delta}{3J},\nonumber\\
\mathcal{I}_4\leq&2\iint_{(\S+\times \S+)\cup(\Sm\times \Sm)\cup(\Sc^+\times \Sc^+)\cup(\Sc^-\times \Sc^-)}n(x)n(y)\log^+|x-y|dxdy\nonumber\\
\leq &C\iint_{(\S+\times \S+)\cup(\Sm\times \Sm)\cup(\Sc^+\times \Sc^+)\cup(\Sc^-\times \Sc^-)}n(x)n(y)(1+|x|^2+|y|^2)dxdy\nonumber\\
\leq &C(M^2+M\int n|x|^2dx).\label{I3 I4 2}
\end{align}
Combining \eqref{I1 I2 -I3 I4 2}, \eqref{K 2} and \eqref{I3 I4 2} yields
\bel\ba\label{T1_2}
\T_1
\geq &CM^2\log\frac{2\delta}{3J}-C(1+M^2+M\int n |x|^2dx).
\ea\eel
We estimate the $\mathcal{I}_2$ term in \eqref{T1T2_2} using \eqref{negative part of entropy} and the second-moment bound of $n$ as in the proof of \eqref{entropy bound}
\begin{align}\label{I_2}
\T_2(t)\leq C(M,A,T)<\infty,\quad \forall t\leq T.
\end{align}
Combining \eqref{T1T2_2}, \eqref{T1_2}, \eqref{I_2} and the second moment bound of $n$, we obtain that 
\bel\ba
S[n](T)
\leq  \frac{1}{(1-\frac{ K}{8\pi})}\bigg(E[n_0]+C(M,A,T)
-CM^2\log\frac{2\delta}{3J}\bigg)
<\infty, \qquad \forall T<\infty.
\ea\eel
Once the entropy is bounded for any finite time, the existence is guaranteed by Proposition \ref{pro 2.1} and Proposition \ref{pro 1.2}.
\end{proof}

\appendix
\section{}
In the appendix, we prove the two local existence theorems stated in section 2.2. The proof follows the same line as the analysis in \cite{BCM2008}.

\subsection{Proof of proposition \ref{pro 2.1}}
\begin{proof}
For any fixed positive $\ep$, following the argument as in section 2.5 of \cite{BDP2006}, we obtain the global solution in $L^2([0,T], H^1)\cap C([0,T], L^2)$ for the regularized Keller-Segel system with advection (\ref{rPKS}).

The goal is to use the Aubin-Lions Lemma to show that the solutions $\{n^\ep\}_{\ep\geq 0}$ is precompact in certain topology. Same as in the paper \cite{BCM2008}, we divide the the proof in steps.

\textbf{Step 1.} A priori estimates on $n^\ep$ and $c^\ep$. In this step, we derive several estimates which we will need later.

First we estimate the second moment
\be
V:=\int_{\rr^2}n^\ep |x|^2dx.
\ee
The time evolution of $V$ can be estimated using the regularised equation \eqref{rPKS} and the fact that the gradient of the regularized kernel $K^\ep$ is bounded $\displaystyle{|\na K^\ep (z)|\leq \frac{1}{2\pi |z|}}$:
\begin{align*}
\frac{d}{dt}V= &4M+\iint_{\rr^2\times\rr^2}n^\ep(x,t)n^\ep(y,t)(x-y)\na K^\ep(x-y)dxdy+2\int x\cdot \bb n^\ep(x)dx\\
\leq& 4M+2C\int  n^\ep |x|^2 dx.
\end{align*}
By Gronwall, we have that
\bel\label{V bound}
V(t)\leq \frac{4M}{2C}e^{2CT}+V_0\leq C_V(T)<\infty,\qquad \forall 0\leq t\leq T,
\eel
from which we obtain that $(1+|x|^2)n^\ep \in L^\infty([0,T],L^1)$ uniformly in $\ep$.

Next, we estimate $|n^\ep\log n^\ep|_{L^\infty([0,T];L^1)}$ and $|\int n^\ep c^\ep dx|_{L_t^\infty}$. Combining the assumption of the proposition \ref{pro 2.1} and \eqref{negative part of entropy} yields
\bel\label{nlogn L1}
\int |n^\ep\log n^\ep | dx\leq \int n^\ep(\log n^\ep +|x|^2)dx+2\log(2\pi)M+\frac{2}{e}.
\eel
Therefore, we proved that $n^\ep \log n^\ep \in L^\infty([0,T], L^1)$ uniformly in $\ep$. Recalling the boundedness of the second moment \eqref{V bound} and the representation of $c^\ep$ as $c^\ep =K^\ep* n^\ep$,  we deduce the following estimate using the Young's inequality $ab\leq e^{a-1}+b\ln b$ for $\forall a,b\geq 1$,
\begin{equation}
|c^\ep|(x)=\bigg|\int K^\ep(x-y)n^\ep(y)dy\bigg|\leq C(M,V,|n^\ep\log n^\ep|_{L^\infty([0,T], L^1)})+C(M)\log(1+|x|)
\end{equation}
uniformly in $\ep$. Combining this with the mass conservation of $n^\ep$ and the second moment control \eqref{V bound}, we deduce that
\bel\label{n_c_L_t_infty_estimate}
\left| \int n^\ep c^\ep dx\right|_{L^\infty([0,T])}\leq C,
\eel
where $C$ is independent of $\ep$.

Next we derive the main a priori estimate, namely, the $L^2([0,T]\times\rr^2)$ estimate of $\sqrt{n^\ep}\na c^\ep$.
First we calculate the time evolution of $\int n^\ep c^\ep dx$:
\be\ba
\frac{1}{2}\frac{d}{dt}\int n^\ep c^\ep dx=&\int n^\ep_t c^\ep dx\\
=&\int (\de n^\ep -\na\cdot(\na c^\ep n^\ep)-b\cdot \na n^\ep)c^\ep dx\\
=&\int n^\ep \de c^\ep dx+\int n^\ep|\na c^\ep|^2dx+\int \bb n^\ep\cdot \na c^\ep dx.
\ea\ee
Integrating this in time yields:
\begin{eqnarray}\nonumber
\lefteqn{\int_0^T\int n^\ep|\na c^\ep|^2 dxdt} \\
& & \ds =\frac{1}{2}\int n^\ep c^\ep dx(T)-\frac{1}{2}\int n^\ep c^\ep dx(0) -\int_0^T\int n^\ep \de c^\ep dxdt-\int_0^T\int n^\ep \bb \cdot \na c^\ep dxdt \label{pa t n c}
\end{eqnarray}
Now we estimate the right hand side of \eqref{pa t n c}. We see from \eqref{n_c_L_t_infty_estimate} that the first two terms on the right hand side is bounded. Next we estimate the third term on the right hand side of \eqref{pa t n c}, which requires information derived from the entropy bound. By the property that $\na\cdot \bb=0$, we formally calculate the time evolution of $S[n](t)$ as follows,
\bel\label{pa t S}
\frac{d}{dt}S[n](t)=-4\int |\na\sqrt{n}|^2dx+\int n^2(t)dx.
\eel
The interested reader is referred to \cite{BDP2006} for more details. We need to estimate the second term in \eqref{pa t S}. Before doing this, note that for $K>1$,
\bel\label{eta K}
\int_{n\geq K}ndx\leq \frac{1}{\log(K)} \int n_+\log ndx \leq \frac{C}{\log(K)}=:\eta(K)
\eel
can be made arbitrarily small. Now we can use the Gagliardo-Nirenberg-Sobolev inequality together with \eqref{eta K} to estimate the second term in \eqref{pa t S}  as follows:
\begin{align*}
\int n^2dx\leq& MK+\int_{n\geq K}n^2 dx \leq MK+\bigg(\int_{n\geq K}n dx\bigg)^{1/2} |n|_3^{3/2}\\
\leq &MK+\eta(K)^{1/2}C M^{1/2}|\na \sqrt{n}|_2^{2}.
\end{align*}
Combining this with \eqref{pa t S} and \eqref{eta K}, we have
\bel
\frac{d}{dt}S[n](t)=-(4-\eta(K)^{1/2}CM^{1/2})\int |\na\sqrt{n}|^2dx+MK.
\eel
The factor $-(4-\eta(K)^{1/2}CM^{1/2})$ can be made non-positive for $K$ large enough and therefore we have that
\be\ba
\int_0^T\int|\na\sqrt{n}|^2 dxdt\leq& \frac{S[n](0)-S[n](T)+MKT}{(4-2 \eta(K)^{1/2}M^{1/2}C)}.
\ea\ee
It follows that $\na\sqrt{n}$ is bounded in $L^2([0,T]\times\rr^2)$. The derivation for $|\na \sqrt{n^\ep}|_{L^2([0,T];L^2)}\leq C$ is similar but more technical, and the interested readers are referred to \cite{BDP2006} for more details.
As a consequence of the $L^2([0,T]\times\rr^2)$ estimate on $\na \sqrt{n^\ep}$ and of the computation
\be
\frac{d}{dt}S[n^\ep](t)=-4\int |\na\sqrt{n^\ep}|^2dx+\int n^\ep (-\de c^\ep)dx,
\ee
we have the estimate
\begin{equation}\label{n_de_c_time_integral_bound}
\int_0^T\int n^\ep( -\de c^\ep )dxdt\leq C.
\end{equation}
This completes  the treatment of the third term on the right hand side of \eqref{pa t n c}.  Next we estimate the last term $\int_0^T\int \bb n^\ep\cdot \na c^\ep dxdt$ in \eqref{pa t n c}. First we calculate the time evolution of $\int G n^\ep dx$:
\be\ba
\frac{d}{dt}\int G n^\ep dx=&\int G \na\cdot(\na n^\ep -\na c^\ep n^\ep-\bb n^\ep)dx\\
=&\int \de G n^\ep dx +\int \na G\cdot \na c^\ep n^\ep dx+\int |\bb|^2n^\ep dx\\
=&\int \bb\cdot \na c^\ep n^\ep dx+\int|\bb|^2n^\ep dx.
\ea\ee
Now integrating in time, we obtain
\bel
\bigg|\int_0^T\int \bb\cdot\na c^\ep n^\ep dx\bigg|\leq \bigg|\int G n^\ep dx(0)\bigg|+\bigg|\int G n^\ep dx(T)\bigg|+\bigg|\int_0^T\int |b|^2n^\ep dxdt\bigg|.\label{b na c n}
\eel
From the assumption $|\bb|\lesssim |x|$, we have that  the right hand side of \eqref{b na c n} can be bounded in terms of the second moment $V$:
\begin{equation}
\bigg|\int G n^\ep dx\bigg|+ \bigg|\int |\bb|^2 n^\ep dx \bigg|\leq C \int n^\ep |x|^2 dx.
\end{equation}
Since the second moment $V$ is bounded \eqref{V bound}, we have that
\begin{equation}\label{last_term_in_pa_t_n_c}
\left|\int_0^T\int \bb \cdot \na c^\ep n^\ep dx\right|\leq C.
\end{equation}
Applying estimates \eqref{n_c_L_t_infty_estimate}, \eqref{n_de_c_time_integral_bound} and \eqref{last_term_in_pa_t_n_c} to \eqref{pa t n c}, we obtain
\begin{equation}\label{sqrt_n_na_c_L_t2_Lx2_bound}
\int_0^T\int n^\ep|\na c^\ep |^2dxdt\leq C<\infty.
\end{equation}
This concludes our first step.

\textbf{Step 2-} Passing to the limit. As in \cite{BCM2008}, the following Aubin-Lions compactness lemma is applied:
\begin{lem}[Aubin-Lions lemma]\cite{BCM2008}
Take $T> 0$ and $1<p<\infty$. Assume that $(f_n)_{n\in \mathbb{N}}$ is a bounded sequence of functions in $L^p([0,T]; H)$ where $H$ is a Banach space. If $(f_n)_n$ is also bounded in $L^p([0,T];V)$ where $V$ is compactly imbedded in $H$ and $\pa f_n/\pa_t \in L^p([0,T];W)$ uniformly with respect to $n\in \mathbb{N}$ where $H$ is imbedded in $W$, then $(f_n)_{n\in \mathbb{N}}$ is relatively compact in $L^p([0, T]; H)$.
\end{lem}
Our goal now is to find the appropriate spaces $V,H,W$ for $n^\ep$. We subdivide the proof into steps, each step determines one space in the lemma.

\textbf{Space $H$: Bound on $|n^\ep|_{L^2([0,T], L^2)}$:} We can estimate the $|n^\ep|_2^2$ by applying the following decomposition trick:
\begin{equation}\label{n ep decomposition}
n^\ep=(n^\ep-K)_++\min\{n^\ep, K\}.
\end{equation}
The second part in \eqref{n ep decomposition} is bounded in $L^p,1\leq p\leq \infty$. The first part is bounded in $L^1$ by
\begin{align}\label{n-K+ 1}
|(n^\ep-K)_+|_1\leq \frac{|n^\ep\log n^\ep|_1}{\log K}=\eta(K),
\end{align}
which can be made arbitrary small. Now we estimate the time evolution of $\int (n^\ep-K)_+^pdx,1< p<\infty$ as follows:
\begin{align}
\frac{1}{p}\frac{d}{dt}\int& (n^\ep-K)_+^pdx\nonumber\\
=&\int (n^\ep-K)_+^{p-1}(\de n^\ep -\na\cdot(\na c^\ep n^\ep)-\bb\cdot \na n^\ep)dx\nonumber\\
=&-\frac{4(p-1)}{p^2}\int |\na(n^\ep-K)_+^{p/2}|^2dx-\int (n^\ep-K)_+^{p-1}\na\cdot(\na c^\ep n^\ep)dx-\int (n^\ep-K)_+^{p-1}\bb\cdot \na (n^\ep-K)_+ dx\nonumber\\
\leq&-\frac{4(p-1)}{p^2}\int |\na (n^\ep-K)_+^{p/2}|^2dx-\frac{p-1}{p}\int (n^\ep-K)_+^p\de c^\ep dx-K\int (n^\ep-K)_+^{p-1}\de c^\ep dx\nonumber\\
=:&-\frac{4(p-1)}{p^2}\int |\na (n^\ep-K)_+^{p/2}|^2dx+\T_1+\T_2.\label{n ep-K+ Lp}
\end{align}
Note that since $\na\cdot \bb\equiv 0$, the term involving $\bb$ vanishes. For the $T_1$ term in \eqref{n ep-K+ Lp}, using the facts that $\de c^\ep=\de K^\ep*n^\ep$ and $|\de K^\ep|_1$ is uniformly bounded,  we can estimate it as follows:
\begin{align}
\T_1\leq& \int (n^\ep-K)_+^p  |\de K^\ep|*(n^\ep-K)_+dx+K\int (n^\ep-K)_+^p dx|\de K^\ep|_1\nonumber\\
\leq& |(n^\ep-K)_+|_{p+1}^{p}||\de K^\ep|*(n^\ep-K)_+|_{p+1}+CK\int (n^\ep-K)_+^pdx\nonumber\\
\leq& C|(n^\ep-K)_+|_{p+1}^{p+1}+CK|(n^\ep-K)_+|_p^p.\label{T1_nep-K+_Lp}
\end{align}
Similarly, we can estimate the $T_2$ term in \eqref{n ep-K+ Lp} as follows:
\begin{equation}
\T_2\leq CK|(n^\ep-K)_+|_{p}^{p}+CK^2|(n^\ep-K)_+|_{p-1}^{p-1}.\label{T2_nep-K+_Lp}
\end{equation}
Combining \eqref{n ep-K+ Lp}, \eqref{T1_nep-K+_Lp} and \eqref{T2_nep-K+_Lp}, we obtain
\begin{align}
\frac{1}{p}\frac{d}{dt}\int &(n^\ep-K)_+^pdx\leq -\frac{4(p-1)}{p^2}\int |\na (n^\ep-K)_+^{p/2}|^2dx\nonumber\\
&+ C|(n^\ep-K)_+|_{p+1}^{p+1}+CK|(n^\ep-K)_+|_p^p+CK^2|(n^\ep-K)_+|_{p-1}^{p-1}.\label{time_evolution_of_nep-K+_Lp}
\end{align}
For the highest order term $C|(n^\ep-K)_+|_{p+1}^{p+1}$ in \eqref{time_evolution_of_nep-K+_Lp}, we use the following Gagliardo-Nirenberg-Sobolev inequality:
\begin{equation}
\int_{\rr^2} f^{p+1}dx\leq C \int_{\rr^2}|\na (f^{p/2})|^2dx\int_{\rr^2}f dx,\quad f\geq 0
\end{equation}
together with \eqref{n-K+ 1} to estimate it as follows
\begin{align}\label{estimate on the highest nonlinearity}
|(n^\ep-K)_+|_{p+1}^{p+1}\leq C |\na((n^\ep-K)_+^{p/2})|_2^2|(n^\ep-K)_+|_1\leq C\eta(K)|\na((n^\ep-K)_+^{p/2})|_2^2.
\end{align}
We can take $K$ big such that it is absorbed by the negative dissipation term in \eqref{time_evolution_of_nep-K+_Lp}. Now applying H\"{o}lder's inequality, Young's inequality and Gronwall inequality to \eqref{time_evolution_of_nep-K+_Lp}, we have that
\begin{align*}
|(n^\ep-K)_+|_{L^p}(t)\leq C(T)<\infty,\quad t\in[0,T], p\in (1,\infty).
\end{align*}
Applying standard argument, see e.g., \cite{BDP2006} proof of Proposition 3.3, we obtain the estimate
\begin{equation}\label{Lp_estimate_of_n_ep}
|n^\ep|_{L^\infty([0,T];L^p)}\leq C(T),\quad p\in(1,\infty).
\end{equation}
In particular, we set $p=2$ and obtain that
\begin{equation}\label{L2L2_estimate_of_n_ep}
|n^\ep|_{L^2([0,T];L^2)}\leq C(T).
\end{equation}
We conclude this step by setting $H:= L^2(\rr^2)$.

\textbf{Space $V$: Bound on $|\na n^\ep |_{L^2([0,T]\times\rr^2)}$:}
First we calculate the time evolution of the quantity $|n^\ep|_2^2$:
\be\ba
\frac{d}{dt}\int |n^\ep|^2dx
=&-2\int |\na n^\ep|^2dx+2\int \na n^\ep \cdot \na c^\ep n^\ep dx-\int \na\cdot(\bb(n^\ep)^2)dx\\
=&-2\int |\na n^\ep|^2dx+2\int \na n^\ep \cdot \na c^\ep n^\ep dx.
\ea\ee
Now integrating in time, we obtain the estimate
\begin{align}
\int |n^\ep|^2&dx(T)-\int |n^\ep|^2dx(0)\nonumber\\
=&-2\int_0^T\int |\na n^\ep|^2dxdt+2 \int_0^T\int n^\ep (\na n^\ep\cdot\na c^\ep)dxdt\nonumber\\
\leq&-2\int_0^T\int |\na n^\ep|^2dxdt+2\left(\int_0^T\int |n^\ep\na c^\ep|^2dxdt\right)^{1/2}\left(\int_0^T\int |\na n^\ep|^2dxdt\right)^{1/2}\label{na n L2 tx}
\end{align}
The terms on the left hand side of \eqref{na n L2 tx} are bounded due to \eqref{Lp_estimate_of_n_ep}. For the last term on the right hand side, we can estimate it as follows. The Hardy-Littlewood-Sobolev inequality yields
\bel\label{na c L4 control}
|\na c^\ep|_4\leq C|n^\ep |_{4/3},
\eel
which implies
\be
|n^\ep \na c^\ep |_2\leq |n^\ep |_{4}|\na c^\ep |_4\leq C|n^\ep |_{4} |n^\ep|_{4/3}.
\ee
Combining this and the $L^p$ bound \eqref{Lp_estimate_of_n_ep} yields the boundedness of $n^\ep\na c^\ep $ in $L^\infty([0,T],L^2)$. Applying this fact and \eqref{Lp_estimate_of_n_ep} in \eqref{na n L2 tx} and set $X=(\int_0^T\int |\na n^\ep|^2dxdt)^{1/2}$, we obtain
\be
X^2-2|n^\ep\na c^\ep|_{L^2((0,T)\times \rr^2)} X\leq \int |n^\ep|^2dx(T)-\int |n^\ep|^2dx(0)\leq C.
\ee
As a result,
\begin{equation}\label{na_nep_L2L2}
|\na n^\ep|_{L^2([0,T]\times\rr^2)}\leq C.
\end{equation}

Same as in \cite{BCM2008}, we set $V:=H^1(\rr^2)\cap \{n||x|n^2\in L^1\}$, which is shown to be compactly imbedded in $H$ there. Thanks to the bound \eqref{L2L2_estimate_of_n_ep}, \eqref{na_nep_L2L2} and the \eqref{V bound}, we have that $n^\ep\in L^2([0,T];V)$.

\textbf{Space $W$: Bound for the $\pa_t n^\ep$:} In order to estimate the $L^2([0,T];H^{-1})$ norm of the function $\pa_t n^\ep$, we first need to get a bound on the fourth moment of the solution $V_4:=\int n( x_1^4+x_2^4)dx$. The time evolution of $V_4$ can be formally estimated as follows:
\begin{align}
\frac{d}{dt}V_4=&\int (\de n-\na \cdot (\na cn)-\bb\na n)(x_1^4+x_2^4)dx\nonumber\\
=&12\int n |x|^2dx-\frac{1}{4\pi}\iint\frac{4(x_1^3-y_1^3)(x_1-y_1)+4(x_2^3-y_2^3)(x_2-y_2)}{|x-y|^2}n(x)n(y)dxdy\nonumber\\
 &+\int \bb n\cdot (4x_1^3, 4x_2^3)dx\nonumber\\
\leq& C(M)\int n|x|^2dx +C\int n(x_1^4+x_2^4)dx.
\end{align}
Combining the second moment estimate \eqref{V bound} and the Gronwall inequality, we have that
\begin{equation}
\int n(x,t) |x|^4dx\leq C, \enskip \quad t\in [0,T].\label{V4 control}
\end{equation}
One can adapt this calculation to the regularized solution $n^\ep$ without much difficulty. We leave the details to the interested reader.

Now we can estimate the $L^2([0,T]; H^{-1})$ norm of the $\pa_t n^\ep$. Combining the $L^p$ bound on $n^\ep$ \eqref{Lp_estimate_of_n_ep}, the bound on $\na c$ \eqref{na c L4 control} and the fourth moment control \eqref{V4 control} and testing the equation \eqref{rPKS} with $f\in L^2([0,T],H^1(\rr^2))$, we have  \myr{DISPLAY below...}
\begin{align*}
\langle \partial_t n^\epsilon, f\rangle_{L^2([0,T]\times\rr^2)}\leq &|\na n^\epsilon|_{L^2([0,T]; L^2)}|f|_{L^2([0,T]; H^1)}+|\na c^\epsilon n^\epsilon|_{L^2([0,T];L^2)}|f|_{L^2([0,T]; H^1)}\\
&+|\bb n^\epsilon|_{L^2([0,T];L^2)}|f|_{L^2([0,T]; H^1)}\\
\leq&|\na n^\epsilon|_{L^2([0,T]; L^2)}|f|_{L^2([0,T]; H^1)}+T^{1/2}|\na c^\epsilon|_{L^\infty([0,T]; L^4)}| n^\epsilon|_{L^\infty([0,T];L^4)}|f|_{L^2([0,T]; H^1)}\\
&+CT^{1/2}\sup_t V_4^{1/4}|n^\epsilon|_{L^\infty([0,T];L^3)}^{3/4}|f|_{L^2([0,T]; H^1)}\\
\leq&C|f|_{L^2([0,T]; H^1)}.
\end{align*}
As a result, we have that $\pa_t n^\ep$ is uniformly bounded in $L^2([0,T]; H^{-1})$.

Combining the results from all the steps above and the Aubin-Lions lemma, we have that $(n^\ep)_\ep$ is precompact in $L^2([0,T];L^2)$. We denote $n$ as the limit of one converging subsequence $(n^{\ep_k})_{\ep_k}$. Moreover, combining \eqref{na c L4 control} and the $L^{4/3}([0,T]\times\rr^2)$ bound on $n^\ep$ which can be derived from \eqref{Lp_estimate_of_n_ep}, we have that $n^{\ep_k} \na c^{\ep_k}$ converge to $n\na c$ in distribution sense.

\textbf{Step 3-} Free energy estimates.
By the convexity of the functional $n\rightarrow \int_{\rr^2}|\na \sqrt{n}|^2dx$, the fact that $n^{\ep_k} \na c^{\ep_k}$ converge to $n\na c$ in distribution sense and weak semi-continuity, we have

\be\ba
\iint_{[0, T]\times \rr^2}|\na\sqrt{n}|^2dxdt\leq&\liminf_{k\ra \infty}\iint_{[0, T]\times\rr^2}|\na\sqrt{n^{\ep_k}}|^2dxdt,\\
\iint_{[0, T]\times \rr^2}n|\na c|^2dxdt\leq&\liminf_{k\ra \infty}\iint_{[0, T]\times\rr^2}n^{\ep_k}|\na c^{\ep_k}|^2dxdt.
\ea\ee
Moreover, it can be checked that $S[n^\ep](t)\rightarrow S[n](t)$ for almost every t, whose proof is similar to the one used in \cite{BDP2006} Lemma 4.6.

Next we show the free energy estimate (\ref{free energy estimate}) using the strong convergence of $\{n^{\ep_k}\}$ in $L^2([0,T]\times \rr^2)$. The key is to show the following entropy dissipation term is lower semi-continuous for the sequence $(n^{\ep_k})$:

\begin{align}
\int_0^T\int n^{\ep_k}&|\na \log n^{\ep_k}-\na c^{\ep_k} -\bb|^2dxdt\nonumber\\
=&4\iint_{[0, T]\times \rr^2}|\na\sqrt{n^{\ep_k}}|^2dxdt+\iint_{[0, T]\times \rr^2}n^{\ep_k}|\na c^{\ep_k}|^2dxdt+\iint_{[0, T]\times \rr^2}n^{\ep_k}|\bb|^2dxdt\nonumber\\
&-2\iint_{[0, T]\times \rr^2} (n^{\ep_k})^2dxdt-2\iint_{[0, T]\times \rr^2} n^{\ep_k} \bb\cdot \na c^{\ep_k} dxdt\nonumber\\
=:&\T_1+\T_2+\T_3+\T_4+\T_5.\label{free energy dissipation}
\end{align}
For the sake of simplicity, later we use $n^\ep$ to denote $n^{\ep_k}$.

First, we estimate the $\T_3$ term in \eqref{free energy dissipation}.
By the Fatou Lemma, we have the following inequality:
\bel\label{Fatou}
\iint_{[0,T]\times \rr^2}n|\bb|^2dxdt\leq \liminf_{\ep_k\rightarrow 0}\iint_{[0,T]\times \rr^2}n^{\ep_k}|\bb|^2dxdt.
\eel
This finises the treatment of $\T_3$.

Next, we show that the term $\T_5$ in \eqref{free energy dissipation} actually converges as $\ep_k\rightarrow 0$. We decompose the difference between $\T_5$ and its formal limit into two parts
\begin{align}
\bigg|\int_0^T\int_{\rr^2}&\bb\cdot \na c n-\bb\cdot \na c^\ep n^\ep dxdt\bigg|\nonumber\\
\leq &\int_0^T\int_{\rr^2}|\na c \cdot \bb n -\na c^\ep\cdot \bb n+n \bb\cdot\na c^\ep-\bb\cdot \na c^\ep n^\ep|dxdt\nonumber\\
\leq &\int_0^T\int_{\rr^2}|\na c-\na c^\ep|\cdot|\bb n|dxdt+\int_0^T\int_{\rr^2}|\bb n-\bb n^\ep|\cdot|\na c^\ep|dxdt\nonumber\\
=:& \T_{51}+\T_{52}.\label{T 5 1 T 5 2}
\end{align}

For the first term $T_{51}$ in \eqref{T 5 1 T 5 2}, applying H\"{o}lder and Hardy-Littlewood-Sobolev inequality, we can estimate it as follows
\be\ba
\T_{51}\leq & \bigg(\int_0^T|\na(c-c^\ep)|_{4}^4dt\bigg)^{1/4}\bigg(\int_0^T\int_{\rr^2}|b\sqrt{n}|^2dxdt\bigg)^{1/2}\bigg(\int_0^T\int_{\rr^2}\sqrt{n}^4dxdt\bigg)^{1/4}\\
\leq & C\bigg(\int_0^T|n-n^\ep|_{{4/3}}^4dt\bigg)^{1/4}\bigg(\int_0^T\int_{\rr^2}|x|^2ndxdt\bigg)^{1/2}\bigg(\int_0^T\int_{\rr^2}n^2dxdt\bigg)^{1/4}\\
\leq & C\bigg(\int_0^T|n-n^\ep|_{2}^2\, |n-n^\ep|_{1}^2dt\bigg)^{1/4}\bigg(\int_0^T\int_{\rr^2}|x
|^2ndxdt\bigg)^{1/2}\bigg(\int_0^T\int_{\rr^2}n^2dxdt\bigg)^{1/4}\\
\leq & C\bigg(\int_0^T|n-n^\ep|_{2}^2dt\bigg)^{1/4}\sup_t|n-n^\ep|_{1}^{1/2}\bigg(\int_0^T\int_{\rr^2}n|x
|^2dxdt\bigg)^{1/2}\bigg(\int_0^T\int_{\rr^2}n^2dxdt\bigg)^{1/4}.
\ea\ee
From the $L^2([0,T]\times \rr^2)$ strong convergence of $(n^\ep)_\ep$, we have that the first factor goes to zero. Other factors are bounded due to \eqref{V bound}, Fatou's Lemma, $L^1$ bound on $n^\ep, n$ and $n\in L^2([0,T]\times\rr^2)$. As a result, $T_{51}$ converges to zero.
Next we estimate the $\T_{52}$ term in \eqref{T 5 1 T 5 2}. Applying the fact that $|\sqrt{|c|}-\sqrt{|a|}|\leq \sqrt{|c-a|}$, the H\"{o}lder and Hardy-Littlewood-Sobolev inequality, we have
\be\ba
\T_{52}
=&\int_0^T\int_{\rr^2}|\bb(\sqrt{n}^2-\sqrt{n^\ep}^2)| \cdot|\na c^\ep|dxdt\\
\leq & \int_0^T\int_{\rr^2}|\bb\sqrt{n}(\sqrt{n}-\sqrt{n^\ep})|\cdot|\na c^\ep|dxdt+ \int_0^T\int_{\rr^2}|\bb\sqrt{n^\ep}(\sqrt{n}-\sqrt{n^\ep})|\cdot |\na c^\ep|dxdt\\
\leq &|\na c^\ep|_{L^4([0,T]\times \rr^2)}|b\sqrt{n}|_{L^2([0,T]\times\rr^2)}\bigg(\int_0^T\int_{\rr^2}\sqrt{|n-n^\ep|}^4dxdt\bigg)^{1/4}\\
&+|\na c^\ep|_{L^4([0,T]\times \rr^2)}|b\sqrt{n^\ep}|_{L^2([0,T]\times\rr^2)}(\int_0^T\int_{\rr^2}\sqrt{|n-n^\ep|}^4dxdt)^{1/4}\\
\leq &C\sup_{t}{|n^\ep|_{L^1(\rr^2)}}^{1/2}|n^\ep|^{1/2}_{L^2([0,T]\times\rr^2)}\bigg(\int_0^T\int (n^\ep+n)|x|^2dxdt\bigg)^{1/2}|n-n^\ep|_{L^{2}([0,T]\times\rr^2)}^{1/2}
\ea\ee
By the same reasoning as in $T_{51}$, we have that this term goes to zero. This finishes the treatment for the term.

Now for the $\T_1, \T_2,\T_4$ terms in (\ref{free energy dissipation}), we can handle them in the same way as in \cite{BDP2006}. Combining all the estimates above we have that
\begin{equation}
\int_0^T\int n|\na \log n-\na c -\bb|^2dxdt\leq
\liminf_{\ep_k\rightarrow0}\int_0^T\int n^{\ep_k}|\na \log n^{\ep_k}-\na c^{\ep_k} -\bb|^2dxdt
\end{equation}

The remaining part of the proof is the same as the one in \cite{BDP2006}.
\end{proof}

\subsection{Proof of proposition \ref{pro 1.2}}
The proof of the proposition follows along the lines of \cite{BCM2008}.

\medskip\noindent
{\textbf{Acknowledgment.} Research was supported by NSF grants DMS16-13911, RNMS11-07444 (KI-Net), ONR grant N00014-1812465 ONR and the Sloan research fellowship FG-2015-66019.
 SH thanks the Ann G. Wylie Dissertation Fellowship and
 Jacob Bedrossian and Scott Smith for many fruitful discussions.
 We thank  ETH Institute for Theoretical Studies (ETH-ITS) for the support and hospitality. 

\end{document}